\documentclass[11pt,reqno]{amsart}
\usepackage{geometry}
\usepackage{amssymb}
\usepackage{hyperref}       
\usepackage{url}            
\usepackage{booktabs}       
\usepackage{amsfonts}       
\usepackage{graphicx}
\usepackage{amsmath}
\usepackage{mathtools}
\usepackage{enumitem}
\usepackage{comment}
\usepackage{parskip}
\usepackage[capitalize]{cleveref}
\usepackage{tikz-cd}
\usepackage{subcaption}
\usepackage{listings}
\usepackage{multirow}
\usepackage{comment}
\usepackage{breqn}
\usepackage{mathtools}
\breqnsetup{breakdepth={10}}
\usepackage[hyperref,doi,url=false,style=alphabetic,maxbibnames=99]{biblatex}

\allowdisplaybreaks

\AtEveryBibitem{
	 \clearlist{address}
	  \clearfield{date}
	    \clearfield{isbn}
		 \clearfield{issn}
		  \clearlist{location}
		   \clearfield{month}
		    \clearfield{series}
		    \clearfield{note}
			 
			 \ifentrytype{book}{}{
			   \clearlist{publisher}
			     \clearname{editor}
			  }
}
\theoremstyle{plain}
\newtheorem{thm}{Theorem}
\newtheorem{lem}{Lemma}
\newtheorem{prop}{Proposition}
\newtheorem{cor}{Corollary}

\newtheorem{rmk}{Remark}

\theoremstyle{remark}

\theoremstyle{definition}

\newtheorem{ex}{Example}
\newtheorem*{ex*}{Example}

\newcommand{\N}{\mathbb{N}}
\newcommand{\R}{\mathbb{R}}
\newcommand{\C}{\mathbb{C}}
\renewcommand{\H}{\mathbb{H}}
\newcommand{\set}[1]{\left\{\,#1\,\right\}}
\newcommand{\paren}[1]{\left(#1\right)}
\newcommand{\bracket}[1]{\left[#1\right]}
\newcommand{\abs}[1]{\left|#1\right|}

\newcommand{\grad}{\nabla}

\newcommand{\F}{\mathcal F}

\newcommand{\eps}{\varepsilon}
\newcommand{\ucover}[1]{\widetilde{#1}}
\renewcommand{\phi}{\varphi}

\newcommand{\gammaT}[1]{\gamma|_{[0,#1]}}
\newcommand{\far}{\mathrm{far}}
\newcommand{\tfar}{t_{\far}}
\newcommand{\inner}[2]{\langle #1,#2\rangle}
\newcommand{\DD}{\mathcal D}

\DeclareMathOperator{\distort}{distortion}
\DeclareMathOperator{\leb}{leb}
\DeclareMathOperator{\Var}{Var}
\DeclareMathOperator{\E}{\mathbb{E}}
\newcommand{\Egamma}[1]{\E_{\gamma \sim W_x}\bracket{#1}}

\DeclarePairedDelimiter{\ceil}{\lceil}{\rceil}
\DeclarePairedDelimiter{\floor}{\lfloor}{\rfloor}
\DeclarePairedDelimiter{\norm}{\lVert}{\rVert}
\usepackage{xcolor}
\newcommand{\plabel}[2]{#2}

\begin{document}
\pagenumbering{gobble}
\pagenumbering{arabic}

\title{Reeb flows transverse to foliations}
\author{Jonathan Zung}
\address{Department of Mathematics, Princeton}
\email{jzung@math.princeton.edu}

\begin{abstract}
	Let $\F$ be a co-oriented $C^2$ foliation on a closed, oriented 3-manifold. We show that $T\F$ can be perturbed to a contact structure with Reeb flow transverse to $\F$ if and only if $\F$ does not support an invariant transverse measure. The resulting Reeb flow has no contractible orbits. This answers a question of Colin and Honda. The main technical tool in our proof is leafwise Brownian motion which we use to construct good transverse measures for $\F$; this gives a new perspective on the Eliashberg--Thurston theorem.
\end{abstract}

\maketitle

\begin{section}{Introduction}
	In their seminal work \plabel{fix:1.1}{{\cite{eliashberg_confoliations_1998}}}, Eliashberg and Thurston proved that a $C^2$ foliation of a closed, oriented 3-manifold not homeomorphic to $S^1\times S^2$ may be $C^0$ approximated by positive or negative contact structures. When the foliation is taut, the approximating contact structures are universally tight and weakly symplectically fillable. This theorem, along with its subsequent generalizations~\cite{kazez_c0_2017,bowden_approximating_2016}, serves as a bridge between contact topology and the theory of foliations. In one direction, one can export genus detection results from Gabai's theory of sutured manifolds to the world of Floer homology~\cite{gabai_foliations_1983,ozsvath_holomorphic_2004}. In the other direction, a uniqueness result of Vogel for the approximating contact structure implies that invariants of the approximating contact structure become invariants of the deformation class of the foliation~\cite{vogel_uniqueness_2016, bowden_contact_2016}.

Colin and Honda asked when the approximating contact structure can be chosen so that its Reeb flow is transverse to the foliation. When this is the case, the foliation can be used to control the Reeb dynamics. In particular, a transverse Reeb flow can have no contractible Reeb orbits. A contact form with this property is called \emph{hypertight}. The hypertight condition is useful for defining and computing pseudo-holomorphic curve invariants. A motivating example for us is cylindrical contact homology, an invariant of contact structures that is well defined when the contact structure supports at least one hypertight contact form~\cite{bao_definition_2018, hutchings_s1s1-equivariant_2022, hutchings_cylindrical_2017}. 

Colin and Honda constructed such transverse Reeb flows for sutured hierarchies~\cite{colin_constructions_2005}. They recently extended their methods to finite depth foliations on closed 3-manifolds~\cite{colin_foliations_2018}. In this setting, although the Reeb flow cannot be made transverse to the closed leaf, it is transverse to a related essential lamination and hence has no contractible orbits. The goal of the present paper is to give a complete answer to the existence question for transverse Reeb flows for all $C^2$ foliations.

A closed leaf is an obstruction to a transverse Reeb flow. Suppose that $\Sigma$ is a closed, oriented surface and $\alpha$ is the 1-form representing a contact structure $\xi$. By Stokes' theorem, $\int_\Sigma d\alpha = 0$ and it follows that $d\alpha=0$ at some point on $\Sigma$. The Reeb flow must be tangent to $\Sigma$ at this point. In particular, this implies that foliations with transverse Reeb flows have no compact leaves, and hence are taut.

More generally, an invariant transverse measure for $\F$ is an obstruction to a transverse Reeb flow. We prove that for $C^2$ foliations, this is the only obstruction.
\begin{thm}\label{thm:main}
	Let $\F$ be a co-orientable $C^2$ foliation on a closed, oriented 3-manifold $M$. Then $T\F$ can be perturbed to a contact structure with Reeb flow transverse to $\F$ if and only if $\F$ does not support an invariant transverse measure.
\end{thm}
	When $\F$ is taut, an invariant transverse measure gives rise to a nontrivial class in $H^1(M,\R)$. Therefore, we have the following corollary (c.f. Conjecture 1.5 from \cite{colin_constructions_2005}):
	\begin{cor}
		If $M$ is a closed, oriented 3-manifold with a co-oriented $C^2$ taut foliation and $M\ncong S^1 \times S^2$, then $M$ supports a hypertight contact structure.
	\end{cor}
	\begin{proof}
		If $H^1(M,\R)=0$, then the foliation has no invariant transverse measure and \cref{thm:main} applies. Otherwise, $M$ supports a finite depth taut foliation. In this case, Theorem 3.14 of \cite{colin_foliations_2018} gives the desired hypertight contact structure.
	\end{proof}
	One should think of invariant transverse measures as exceptional. In fact, under some conditions, it follows from a result of Bonatti--Firmo that non-existence of invariant transverse measures is generic for $C^\infty$ foliations.
	\begin{cor}\label{cor:generic}
		Suppose that $M$ a closed, oriented, atoroidal 3-manifold. Then any $C^\infty$ taut foliation on $M$ is $C^0$ close to a hypertight contact structure.
	\end{cor}
	\begin{cor}\label{cor:cylindrical}
		Cylindrical contact homology is an invariant of the taut deformation class of $C^\infty$ taut foliations on closed, oriented, atoroidal 3-manifolds.
	\end{cor}
	We defer explanations of the last two corollaries to \cref{sec:generic}.

	Our method contrasts with prior constructions of contact approximations. The strategy of Eliashberg and Thurston begins with identifying some closed curves in leaves of $\F$ with attracting holonomy. They produce a contact perturbation in the neighbourhood of these curves, and then ``flow'' the contactness to the rest of the 3-manifold. It is not clear how to control the Reeb vector field during this flow operation. In the case of sutured hierarchies, Colin and Honda give an explicit inductive construction of the contact perturbation. At each step of the sutured hierarchy, they can ensure that the Reeb flow has a good standard form near the boundary compatible with the sutures.

	Our construction begins with an arbitrary transverse measure for $\F$ and smooths it out by \emph{logarithmic diffusion}. This process is similar to the leafwise heat flow studied by Garnett~\cite{garnett_foliations_1983}. The advantage of our diffusion process is that we can show that in finite time, the transverse measure becomes \emph{log superharmonic}. Roughly speaking, this means that the transverse distance between two nearby leaves is the exponential of a superharmonic function. This gives a global picture of the holonomy of $\F$. Finally, we show that there is a canonical way to deform a log superharmonic transverse measure into a contact structure. Since this perturbation is done in one shot, we have full control over the Reeb flow.

\end{section}

\begin{section}*{Acknowledgements}
The author would like to thank Peter Ozsv\'{a}th for his encouragement and guidance during this project. The author also benefited from helpful conversations with several mathematicians, including Jonathan Bowden, Vincent Colin, Sergio Fenley, David Gabai, and Rachel Roberts. Thanks to Charles Fefferman for explaining the technique used in \cref{prop:perturbkernel}. Finally, the author is grateful to the anonymous referee for a careful reading and many helpful suggestions. This work was partially supported by NSF grant DMS-1607374.
\end{section}

\begin{section}{Examples}
	In this section, we give some concrete examples of foliations and how one may deform them to contact structures. The reader should keep these examples in mind while reading the rest of the paper. We expect that this section can be read with minimal background, but the reader may wish to review \cref{sec:prel1,sec:prel4} for definitions of foliations and contact structures, and \cref{sec:prel2,sec:prel3} for a discussion of holonomy. 
	
	\begin{ex}[creating contact regions using holonomy]\label{ex:1}
		Consider $\R^3$ foliated by planes $z=\text{const}$ and endowed with the Riemannian metric $g=dx^2+dy^2+2^{-2x} dz^2$. This is a local model for a foliation with holonomy. One may produce a contact structure by rotating the tangent planes to the foliation by a Riemannian angle of $\varepsilon$ around an axis parallel to the $x$ axis. Any line parallel to the $x$ axis is a Legendrian. Travelling along such a Legendrian in the positive $x$ direction, the Riemannian angles are constant, so the Euclidean angles are increasing. This twisting of contact planes along a Legendrian is the hallmark of a contact structure.

		In equations, the contact form is $\alpha:=2^{-x}dz+\varepsilon dy$ and
		\begin{align*}
			d\alpha = -2^{-x}dx\wedge dz
		\end{align*}
		In this case, the Reeb flow is not transverse to the foliation but points in the $y$ direction. This construction may equally well be done in the quotient of $\R^3$ by the isometry $(x,y,z)\mapsto (x+1,y, 2z)$. Here we have seen the general rule that fix:``Legendrians (in the characteristic foliation on a leaf of the foliation) flow in the direction of contracting holonomy''.

		\plabel{fix:4.1}{This kind of construction is visualized in \cref{fig:examples}.} Three boxes are shown, each equipped with both a product foliation transverse to the $z$ direction and an approximating contact structure. The top and bottom faces of each box are leaves. The red dotted lines show the characteristic foliation of the contact structure on some of the boxes' faces. Observe that in the first box, the slope of the characteristic foliation on the left face is greater than the slope on the right face. These ever increasing slopes are unsustainable in a closed manifold. The second and third boxes (corresponding with \cref{ex:1} and \cref{ex:2}) are diffeomorphic to the first, but we have introduced some contraction of the leaves of the foliation. Now the slopes on the left and right faces can be made identical, facilitating gluing up to a closed manifold.
		
	\end{ex}
	\begin{figure}[ht]
		\centering
		\includegraphics[width=\linewidth]{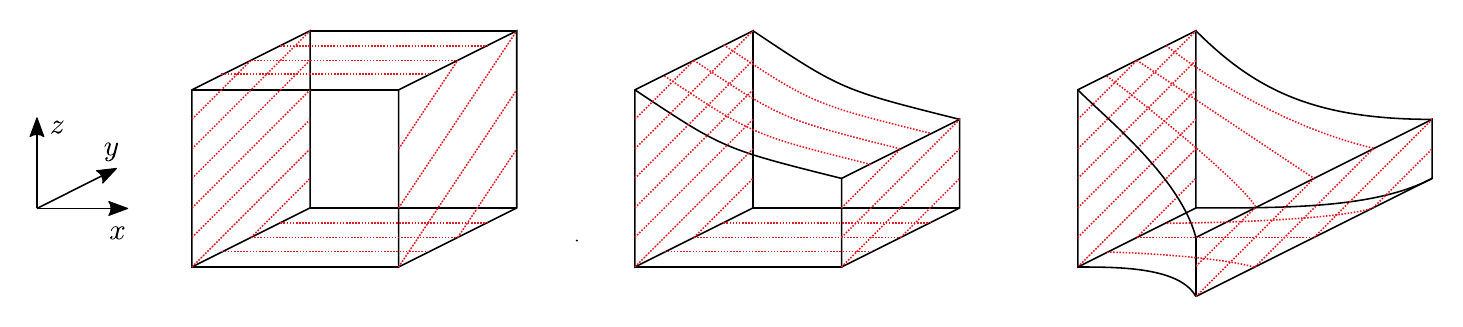}
		\caption{
		}\label{fig:examples}
\end{figure}
	\begin{ex}[Anosov flows]\label{ex:2}
		\cref{ex:1} can be modified so that the Reeb flow is transverse to the leaves. The Legendrian flow will be Anosov, i.e. spreading out in the $y$-direction as well as contracting in the $z$-direction. The relevant metric is $g=dx^2 +2^{2x} dy^2 + 2^{-2x} dz^2$ with the flow parallel to the $x$-axis.

		It is instructive to write this example in different coordinates. Let $$\H=\set{(x,y)\, \mid\, y> 0}$$ be the upper half plane. Consider the manifold $\H \times \R$, foliated by planes of the form $\H \times \set{pt}$. \plabel{fix:4.3}{Endow the manifold $\H \times \R$ with coordinates $x,y,t$}, where $x$ and $y$ correspond with the coordinates on $\H$ and $t$ parameterizes the $\R$ factor. Let $f(x,y,t)=y$.  The 1-form $f dt$ defines the horizontal foliation and is an example of a harmonic transverse measure. The flow parallel to the $y$ axis \plabel{fix:4.2.1}{oriented in the negative $y$ direction is an Anosov flow.} The 1-form $\alpha=fdt - \eps \star d\log f$ is a contact perturbation of our foliation. Indeed,
		\begin{align*}
			\alpha &= y dt + \eps \frac 1 y dx\\
			\alpha \wedge d\alpha &= \eps(2\frac 1 y dx\wedge dy \wedge dt) + O(\eps^2)
		\end{align*}
		This time, the Reeb flow is transverse to the foliation since $d\alpha$ evaluates positively on $\F$. The moral here is that ``spreading of the Legendrians'' in the characteristic foliation can contribute to contactness, and also help to make the Reeb flow transverse to the foliation.

		\plabel{fix:4.2.2}{More generally, the stable/unstable foliations of Anosov flows have no invariant transverse measures, and so are candidates for the application of \cref{thm:main}. However, these stable/unstable foliations are in general only $C^1$ so our theorem does not directly apply.}
	\end{ex}

	\begin{ex}[a coarsely harmonic transverse measure]\label{ex:3}
		Here is an example of a foliation with holonomy that one can keep in mind while reading the rest of the paper.	Let $P$ be a pair of pants with boundary components $a$,$b$, and $c$. Let $\phi$ be a diffeomorphism of $P$ exchanging $b$ and $c$. Let $M^\circ$ be the mapping torus of $\phi$. The manifold $M^\circ$ has a foliation by parallel copies of $P$. It has two toroidal boundary components, each with a horizontal foliation and one ``twice as long'' as the other. Glue these boundary components together by a map $\theta \mapsto 2\theta$, where $\theta$ parameterizes the direction transverse to $P$. Call the resulting foliated manifold $M$.

		Now let $f$ be a function on $P$ which takes the value $1$ on $b$ and $c$ and the value $2$ on $a$. We can further arrange that $f$ is invariant under $\phi$ and has only a single critical point. Then $f$ pulls back to $M^\circ$. After some smoothing near the cuffs, the 1-form $f\, d\theta$ is a smooth transverse measure on $M$. With respect to this transverse measure, the manifold has contracting holonomy along paths from $a$ to $b$ or $a$ to $c$.

		We will return to the contact perturbation in this case in \cref{ex:3revisited} in \cref{sec:farkas}.
	\end{ex}

	\begin{ex}[sutured manifold]
		The foliation in \cref{ex:3} has every leaf dense, but we will also need to consider foliations whose leaves accumulate on sublaminations. The example to keep in mind is a taut sutured manifold \plabel{fix:5.1}{\cite{gabai_foliations_1983}}. One of the simplest sutured manifolds is the solid torus $T_n$ with $2n$ longitudinal sutures, $n\geq 2$. It is foliated by a stack of monkey saddles, each of which accumulates on the boundary leaves. For a natural choice of Riemannian metric, the foliation has contracting holonomy along every path to infinity in a leaf. This foliation map be perturbed into a contact structure whose characteristic foliation on each leaf consists of radial lines emanating from a single elliptic singularity to infinity. See \cref{fig:monkey}. Notice that the curves in the characteristic foliation travel in the direction of contracting holonomy. The Reeb flow has a single periodic orbit along the core of the solid torus; every other orbit enters along a positively oriented boundary leaf and leaves along a negatively oriented one. 
		\begin{figure}[ht]
				\centering
				\includegraphics[width=0.75\linewidth]{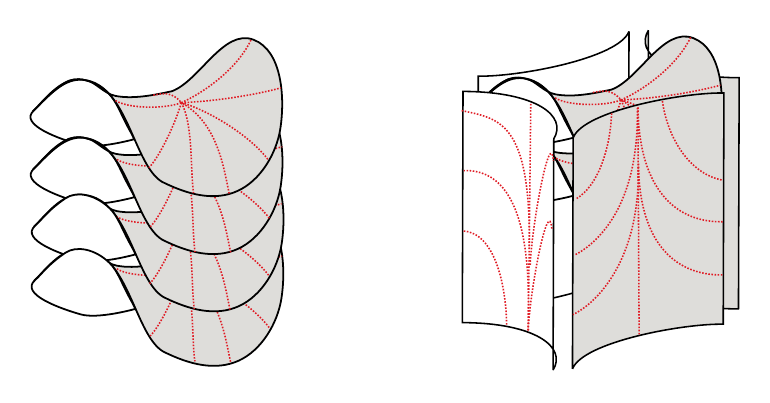}
				\caption{This is a foliation by saddles of the solid torus with four longitudinal sutures. The $S^1$ direction is vertical. The limiting annular boundary leaves are added on the right. The red dotted lines denote the characteristic foliation of an approximating contact structure.
				}\label{fig:monkey}
		\end{figure}
	\end{ex}

\end{section}

\begin{section}{Preliminaries}\label{sec:prel}
	\begin{subsection}{Foliations and their regularity}\label{sec:prel1}
		A \emph{foliation} on a 3-manifold $M$ is a decomposition of $M$ into surfaces, locally modelled on $\R^2 \times \R$. A foliation is specified by a covering of $M$ by charts such that their transition maps preserve the decomposition of $\R^3$ into horizontal planes. The surfaces in the decomposition are called leaves. A foliation is oriented if the transition maps preserve the orientations of the leaves and co-oriented if the transition maps preserve the orientation of the transverse $\R$ factor. 

		For $k\geq l$, we say that a foliation is $C^{k,l}$ if the mixed partial derivatives of the transition maps up to order $l$ in the transverse direction and up to order $k$ in total exist and are continuous. \plabel{fix:6.1}{Calegari showed that a $C^{1,0}$ foliation can be improved to a $C^{\infty,0}$ foliation by a topological isotopy \cite{calegari_leafwise_2001}.} His method extends to show that a $C^{l,l}$ foliation can be improved to a $C^{\infty, l}$ foliation by a $C^l$ isotopy. (His Lemma 3.2 holds for $C^{l,l}$ foliations, so one need only check that Lemma 3.3 in \cite{calegari_leafwise_2001} works using $C^l$ isotopies.) However, it might not be possible to improve the transverse regularity class of a foliation; the holonomy maps of a $C^{k,0}$ foliation are $C^0$ functions which might not be conjugate to smooth functions. A foliation is called $C^l$ when it is $C^{l,l}$. We will also use the notation $C^{k,l}$ to denote the regularity class of functions on $M$.

		\plabel{fix:folconventions}{In this paper, all manifolds are closed and oriented. All foliations are oriented and co-oriented. In light of Calegari's result, we will assume from now on that our foliations are $C^{\infty,2}$ unless stated otherwise.}
	\end{subsection}
	\begin{subsection}{Taut foliations}
		\plabel{fix:6.4}{An oriented $C^{\infty,2}$ foliation $\F$ on an oriented 3-manifold $M$} is \emph{taut} if any of the following equivalent conditions hold:
		\begin{enumerate}
			\item For each point $x\in M$, there is a closed curve transverse to $\F$ passing through $x$.
			\item There is a volume preserving flow transverse to $\F$.
			\item There is a closed 2-form $\omega$ evaluating positively on $T\F$.
		\end{enumerate}
		Taut foliations enjoy a number of good properties.
		\begin{enumerate}
			\item Loops transverse to $\F$ are not contractible.
			\item Leaves of $\F$ are $\pi_1$-injective. 
		\end{enumerate}	
		If we exclude the exceptional case $M\cong S^1\times S^2$ we can say more.
		\begin{enumerate}
			\item $M$ is irreducible.
			\item The universal cover $\ucover M$ is $\R^3$.
			\item Leaves of $\ucover F$ are properly embedded planes in $\ucover M$.
		\end{enumerate}
		If a foliation has no compact leaves, then it is taut.
	\end{subsection}

	\begin{subsection}{Contact structures}\label{sec:prel4}
		A \emph{contact structure} on an oriented 3-manifold $M$ is a plane field $\xi$ such that there exists a 1-form $\alpha$ with $\xi=\ker(\alpha)$ and $\alpha \wedge d\alpha > 0$. The \emph{Reeb flow} of $\alpha$ is a vector field $R$ uniquely defined by the following properties:
		\begin{align*}
			d\alpha(R,-)&=0\\
			\alpha(R)&=1
		\end{align*}
		The Reeb flow is always transverse to $\xi$, and moreover is transverse to any tangent plane on which $d\alpha > 0$. The Reeb flow preserves $\alpha$. In particular, it preserves the volume form $\alpha \wedge d\alpha$. Compare this with the second definition of taut foliation above; it follows that a foliation with a transverse Reeb flow is taut. If $\xi=\ker(\alpha)$ for some $\alpha$ whose corresponding Reeb flow has no contractible closed orbits, we say that $\xi$ is \emph{hypertight}.

		Given a surface $\Sigma \subset M$, the \emph{characteristic foliation} of $\xi$ on $\Sigma$ is the singular codimension 1 foliation on $\Sigma$ defined by the intersection of $\xi$ with $\Sigma$.
	\end{subsection}

	\begin{subsection}{Transverse measures}\label{sec:prel2}
		A \emph{transversal} to a foliation $\F$ is a smooth closed arc positively transverse to $\F$. When $M$ is endowed with a Riemannian metric, we define an \emph{orthogonal transversal} to be a transversal which is orthogonal to $\F$.

		A transverse measure $\tau$ is an assignment of a non-negative real number to each transversal to $\F$. We ask that the assignment be countably additive under concatenation of arcs. 

		Suppose that $I_1$ and $I_2$ are two transversals. \plabel{fix:6.2}{We say that $I_1$ and $I_2$ are homotopic if there is a homotopy $H:[0,1]\times [0,1] \to M$ between them such that $H|_{[0,1]\times \set{0}}=I_1$, $H|_{[0,1]\times \set{1}}=I_2$, $H|_{\set{0}\times [0,1]}$ is contained in a leaf of $\F$, and $H|_{\set{1}\times[0,1]}$ is contained in a leaf of $\F$.}

		$\tau$ is called an \emph{invariant transverse measure} if $\tau(I_1)=\tau(I_2)$ whenever $I_1$ and $I_2$ are homotopic and $\tau$ is not the zero transverse measure.

		A compact leaf $\lambda$ gives rise to an invariant transverse measure which assigns to each transverse arc the number of intersections with $\lambda$. We will generally be concerned with \emph{$C^{k,l}$ transverse measures of full support}. \plabel{fix:6.3}{By this we mean a transverse measure which can be encoded as a nowhere-vanishing $C^{k,l}$ 1-form $\tau$ with $\ker(\tau)=T\F$.} Note that a $C^{\infty,k}$ foliation always admits a $C^{\infty,k-1}$ transverse measure of full support. In a foliation chart with coordinates $x,y,z$, one may take the transverse measure to be $g\, dz$ where $g$ is any positive $C^{\infty,k-1}$ function. 
	\end{subsection}

	\begin{subsection}{Perturbing away transverse measures}\label{sec:generic}
		\plabel{fix:22.1}{The following lemma was observed by Bowden and documented in \cite[Lemma 7.2]{zhang_monopole_2016}:}
		\begin{lem}\label{lem:perturb}
			Let $M$ be an atoroidal 3-manifold and $\F$ a $C^\infty$ taut foliation on $M$. Then $\F$ can be $C^0$ approximated by a $C^\infty$ taut foliation $\F'$ such that either $\F'$ has no transverse invariant measure or the pair $(M,\F')$ is homeomorphic to a surface bundle over $S^1$ foliated by the fibers.
		\end{lem}
		\begin{proof}[Proof sketch]
			On $C^2$ foliations, invariant transverse measures are either supported on compact leaves or have full support. \plabel{fix:22.2}{This is a consequence of Sacksteder's theorem~\cite[, Theorem 8.2.1]{sacksteder_foliations_1965, candel_foliations_2000-1}.} In the former case, if $\F$ is $C^\infty$, a result of Bonatti--Firmo allows us to perturb away compact leaves of genus $\geq 2$~\cite{bonatti_feuilles_1994}. In the latter case, Tischler showed that $(Y,\F)$ is a deformation of a fibration~\cite{tischler_fibering_1970}. See \cite[Lemma 7.2]{zhang_monopole_2016} for more details.
		\end{proof}
		\begin{proof}[Proof of \cref{cor:generic}]
			Given a $C^\infty$ taut foliation $\F$ on an atoroidal manifold $M$, \cref{lem:perturb} yields a new foliation $\F'$ which is $C^0$ close to $\F$ that either has no invariant transverse measure or is a surface bundle over $S^1$. If $\F'$ has no invariant transverse measure, then \cref{thm:main} applies. Suppose instead that $\F'$ is a surface bundle over $S^1$. Since $M$ is atoroidal, the fiber genus is at least 2 and the monodromy is pseudo-Anosov. By \cite[Theorem 3.17]{colin_foliations_2018}, the (unique) contact perturbation of $\F'$ is hypertight.
		\end{proof}
	
		\begin{proof}[Proof of \cref{cor:cylindrical}]
			Vogel proved that if $\F$ is a $C^2$ taut foliation on an atoroidal 3-manifold, all contact structures in a $C^0$ neighbourhood of $T\F$ are isotopic. Moreover, pairs of taut foliations which are homotopic through taut foliations have isotopic contact approximations \cite[Theorem 9.3]{vogel_uniqueness_2016}. By \cref{cor:generic}, this contact approximation is hypertight and hence has well-defined cylindrical contact homology.
		\end{proof}
	\end{subsection}
	
	\begin{subsection}{Holonomy}\label{sec:prel3}
		Let $M$ be an oriented 3-manifold with a taut, co-oriented $C^2$ foliation, $\F$. 

		Let $\gamma\colon [0,T]\to M$ be a path which is contained in a leaf of $\F$ and supported in a single foliation chart $U$. Choose $I_0,I_T$, two transversals to $\F$ contained in $U$ and passing through $\gamma(0),\gamma(T)$ respectively. We define the holonomy map $h_\gamma\colon I_0 \to I_T$ to be the map which preserves the transverse coordinate of the foliation chart, i.e. for all $x \in I_0$, $x$ and $h_\gamma(x)$ lie on the same leaf in $U$. If $\gamma'$ is another path in the same leaf of $U$ and is homotpic rel endpoints to $\gamma$, then then $h_{\gamma}=h_{\gamma'}$. The holonomy map can be defined for arbitrary paths $\gamma$ tangent to $\F$ by composing holonomy maps for short subpaths. As $\gamma$ grows in length, the domain of $h_\gamma$ tends to shrink. From our perspective, we are usually just concerned with the germ at 0 of $h_\gamma$, so this shrinking of the domain does not matter.

		Since $\F$ is a $C^2$ foliation, $h_\gamma$ is $C^2$. Thus, we may take its derivative at $\gamma(0)$ which we call $$h_\gamma':T_{\gamma(0)}(I_0)\to T_{\gamma(T)}(I_T).$$ The embedding of $I_0$ in $M$ gives an identification of $T_{\gamma(0)}(I_0)$ with $T^\perp_{\gamma(0)}\F$, and likewise an identification of $T_{\gamma(T)}I_T$ with $T^\perp_{\gamma(T)}\F$. Via this identification, we may regard $h_\gamma'$ as a map from $T^\perp_{\gamma(0)}\F$ to $T^\perp_{\gamma(T)}\F$. This map depends only on $\gamma$ and not on the choice of $I_0$ and $I_T$. Moreover, if $\gamma$ and $\gamma'$ are two paths in the same leaf of $\F$ and are homotopic rel endpoints, then 
		\begin{equation}\label{eqn:hmtpy}
			h'_{\gamma}=h'_{\gamma'}
		\end{equation}

		Suppose now that $\F$ is equipped with a $C^{\infty,1}$ transverse measure $\tau$, encoded as a 1-form with $\ker(\tau)=T\F$ and which evaluates positively on transversals to $\F$. Then we define $$\abs{h_\gamma'}_\tau=\frac{\tau(h_\gamma'(v))}{\tau(v)}$$ where $v$ is any non-zero vector in $T^\perp_{\gamma(0)}\F$. In other words, $\abs{h_\gamma'}_\tau$ is the factor by which holonomy along $\gamma$ stretches $\tau$-lengths. Note that this norm does not depend heavily on the choice of $\tau$. Indeed, given any two such differentiable transverse measures $\tau, \mu$, we have $\tau = g\mu$ for some strictly positive function $g$ on $M$. \plabel{fix:8.1}{Then for any path $\gamma$ in a leaf of $\F$, one can verify that $$ \inf_{M} g^2 \leq \frac{|h_\gamma'|_\mu}{|h_\gamma'|_\tau} \leq \sup_{M}g^2.$$}

		Given a leaf $\lambda$, \plabel{fix:8.2.1}{let $\widetilde \lambda$ be the universal cover of $\lambda$. At this point, we are thinking of $\ucover \lambda$, not as embedded in $\ucover M$.} Choose a basepoint $b$ in $\ucover \lambda$. We define the Radon-Nikodym derivative of $\tau$ to be the function $f_{\tau,\lambda,b}\colon\ucover{\lambda}\to \R$ determined by the formula
		$$ f_{\tau,\lambda,b}(x)=\abs{h_{\gamma}'}_{\tau},$$
		where $\gamma$ is a path in $\lambda$ which lifts to a path $\ucover \gamma$ in $\ucover \lambda$ from $b$ to $x$. By the homotopy invariance property observed in \cref{eqn:hmtpy}, this definition is independent of the choice of the path $\gamma$. One should think of $f_{\tau,\lambda,b}$ as the transverse distance to a nearby leaf as measured by $\tau$. Different choices of basepoint yield the same function $f_{\tau,\lambda,b}$ up to a constant factor. To be more precise, if $b_1$ and $b_2$ are two choices of basepoint on $\ucover \lambda$ and $\gamma$ is a path in $\lambda$ which lifts to a path in $\ucover \lambda$ from $b_1$ to $b_2$, then
		$$f_{\tau,\lambda,b_1}=|h'_\gamma|_\tau f_{\tau,\lambda,b_2}.$$ When there is no danger of ambiguity, we abbreviate $f_{\tau,\lambda,b}$ to $f_\tau$.

		\begin{rmk}\label{rmk:fsmooth}
			In a foliation chart with coordinates $x,y,z$ and leaves $z=\text{const}$, one may write any $C^{\infty, 1}$ transverse measure $\tau$ as $g\, dz$ for some $C^{\infty,1}$ function $g$. The restriction of $g$ to any given leaf $\lambda$ agrees with $f_{\tau,\lambda,b}$ up to multiplication by a constant depending only on the choice of basepoint. It follows that $f_{\tau,\lambda,b}$ is $C^\infty$ on each leaf.
		\end{rmk}

		A theme in what follows is that for the best transverse measures, $f_{\tau, \lambda,b}$ has no local minima. \plabel{fix:introduceriemann}{Fix a Riemannian metric on $M$.} Each leaf inherits a Riemannian metric from the Riemannian metric we chose on $M$. If $f_{\tau, \lambda, b}$ is a harmonic function on $\ucover \lambda$ for every choice of leaf $\lambda$, we say that $\tau$ is a \emph{harmonic transverse measure}. \plabel{fix:8.5}{By Theorem 1c of \cite{garnett_foliations_1983}, this is equivalent} to more traditional notions of harmonic transverse measure using the leafwise Laplacian operator, at least for differentiable transverse measures like $\tau$. See also \cite{candel_harmonic_2003} and \cite{deroin_random_2007} for more discussion of harmonic transverse measures. When $\log f_{\tau,\lambda,b}$ is superharmonic (resp. strictly superharmonic) on each leaf $\ucover \lambda$, we say that $\tau$ is \emph{log superharmonic} (resp. \emph{strictly log superharmonic}). Since $\log$ is a concave function, every harmonic transverse measure is also log superharmonic.

		While $f_{\tau}$ is defined on $\ucover \lambda$ and makes sense only up to constant factors for each leaf, several related objects descend to $M$. The 1-form $d\log f_{\tau}$ is a well-defined section of $T^*\F$. It measures the infinitesimal rate of contraction or expansion of leaves in directions tangent to $\F$. The function $\Delta \log f_\tau$, where $\Delta$ denotes the leafwise Laplace-Beltrami operator, is consequently also defined on $M$. Finally, the leafwise level sets of $f_\tau$ descend to a singular codimension 2 foliation on $M$.

		\begin{prop}\label{prop:iotav}
			Suppose $\tau$ is a $C^{\infty,1}$ transverse measure of full support. Choose a vector field $v$ transverse to $\F$ with $\tau(v)=1$. Then $\iota_v d\tau=-d\log f_\tau(v) $ as sections of $T^* \F$.
		\end{prop}
		\begin{proof}
			Choose a foliation chart with coordinates $x,y,z$ where $\F = \ker(dz)$. We may further arrange that $v$ is parallel to the $z$ axis. In this coordinate system, $\tau= g \, dz$ for some $C^{\infty,1}$ function $g$. 

			Now
			\begin{align}
				\iota_v d\tau &= \iota_v (dg \wedge dz)\nonumber\\
				&= dg(v)\, dz - dz(v)\, dg\nonumber\\
				&= dg(v) \, dz - \frac 1 g dg\nonumber\\
				\intertext{\rightline{since $\tau(v)=1$}\qquad}
				&= dg(v) \, dz - d\log g \label{eqn:finalans}
			\end{align}

			As noted in \cref{rmk:fsmooth}, $g$ agrees with $f_\tau$ on each leaf up to a constant factor, so $d \log g = d \log f_\tau$ on each leaf. When we restrict \cref{eqn:finalans} to any given leaf $\lambda$, the term $dg(v) \, dz$ vanishes and we are left with $-d\log f_\tau$ as desired.
		\end{proof}
	\end{subsection}
	\begin{subsection}{Leafwise Brownian motion and diffusion}\label{sec:diffusion}
		We provide here a summary of the main properties of Brownian motion that we will use. \plabel{fix:8.6}{We direct the reader to \cite{morters_brownian_2010} for a comprehensive introduction to Brownian motion on $\R^n$. Chapters 3 and 4 of \cite{hsu_stochastic_2002} provide a treatment of the Riemannian case. One may also consult \cite{deroin_random_2007} which specializes to the leafwise Brownian motion that we will consider in this paper.}

		Let $\lambda$ be a complete Riemannian manifold with bounded geometry and let $x$ be a point on $\lambda$. As suggested by the notation, the Riemannian manifold we consider will usually be a leaf of $\F$ with a Riemannian metric inherited from one on $M$. Such a surface always has bounded geometry because $M$ is compact. \plabel{fix:8.4}{Let $\Gamma_x$ be the space of all continuous paths $\gamma:[0,\infty)\to \lambda$ satisfying $\gamma(0)=x$. We equip $\Gamma_x$ with the uniform topology on compact sets of $[0,\infty)$ and will use the induced Borel $\sigma$-algebra.}

		On $\Gamma_x$ there is a probability measure $W_x$ called the \text{Wiener measure}. We will often refer to a path drawn from $W_x$ as a Brownian path. For any function $f$ on $\lambda$, we define its time $t$ diffusion $D^t(f)$ by
		$$D^t(f)(x)=\E_{\gamma \sim W_x}\bracket{f(\gamma(t))}.$$ Let $\Delta$ be the Laplace-Beltrami operator on $\lambda$. The Wiener measure satisfies the following properties:
		\begin{enumerate}
			\item \textbf{Markov property.} Suppose $\gamma$ is drawn from $W_x$. Fix some $t_0 \geq 0$ and any $y\in \lambda$. Then after conditioning on $\gamma(t_0)=y$, $\gamma|_{[t_0,\infty)}$ is distributed like $W_y$. This implies that for any $t,t'\geq 0$, $$D^{t+t'}=D^t\circ D^{t'}.$$
				\item \textbf{Adapted to the metric.} For a real valued $C^2$ function $f_0$ on $\lambda$, $$\frac{\partial }{\partial t}D^t f_0 \bigg|_{t=0}=-\Delta f_0.$$ In other words,	$D^t(f)$ is the time $t$ solution to the heat equation $$\frac {\partial f} {\partial t} = -\Delta f$$ with initial condition $f_0$. 
		\end{enumerate}

		Combining these two properties with linearity of expectation, we find that $\Delta$ commutes with diffusion when applied to a $C^2$ function $f$:
		\begin{align}
			\Delta D^{t'} f &= -\frac {\partial}{\partial t} D^t D^{t'} f\bigg|_{t=0}\\
			&= -D^{t'}\frac {\partial}{\partial t} D^t f\label{eqn:mid}\bigg|_{t=0}\\
			&= D^{t'}\Delta f\label{eqn:lap}
		\end{align}

		\begin{rmk}
			The hypothesis of bounded geometry allows us to disregard paths that leave the manifold in finite time.
		\end{rmk}

		In the course of the paper, we will define another diffusion operator $\DD_{T,R,S}$ acting on transverse measures. It should not be confused with $D^t$.
	\end{subsection}
	\begin{subsection}{Forms and Currents}\label{sec:currents}
		\plabel{fix:19.2}{In the rest of the paper, we will use the language of currents carried by foliations. We refer the reader to \cite{sullivan_cycles_1976} and \cite{dinh_introduction_2005} for more background on currents and foliations.} Informally, an $i$-current on a manifold $M$ is an oriented $i$-dimensional submanifold of $M$ (possibly disconnected, possibly with boundary), or a weak limit of such submanifolds. Similarly, a 2-current carried by a codimension 1 foliation $\F$ on a 3-manifold is informally a weak limit of surfaces contained in leaves of $\F$. 

		Let us define the relevant spaces more carefully and with some attention to regularity. Let $U$ be a foliation chart. Let $x,y,t$ be local coordinates for $U$, where $x$ and $y$ are coordinates for tangential directions and $t$ is the transverse coordinate. Let $A$ be a subset of $\set{dx,dy,dt}$, and let $f$ be a $C^{k,l}$ function on $U$. Let $\chi$ be the indicator function for the subset $A$. Then we say that the differential form $f \bigwedge_{\beta \in A} \beta$ has \emph{adjusted regularity} $(k+|A|, l+\chi(dt))$. For example, if $f$ is $C^{k,l}$, then $f\,dx\wedge dt$ has adjusted regularity $(k+2, l+1)$. The adjusted regularity of a sum of such differential forms is the minimum of the adjusted regularities of its summands. The adjusted regularity of a differential form on $M$ is $(k,l)$ if has adjusted regularity $(k,l)$ in each foliation chart. With this definition, the exterior derivative preserves adjusted regularity.

		In \cref{sec:farkas}, it will be more convenient to work in a Sobolev space. If the function $f$ is in the Sobolev space of functions with $L^2$ norms on any partial derivatives with order at most $l$ in the transverse direction and total order at most $k$, then we say that $f$ has \emph{Sobolev adjusted regularity} $(k,l)$. Similarly, we say that $f \bigwedge_{\beta \in A} \beta$ has \emph{Sobolev adjusted regularity} $(k+|A|, l+\chi(dt))$.

		Now let $\Omega^i(M)$ denote the space of $i$-forms of adjusted regularity $C^{\infty,2}$. Note that $\bigoplus_i \Omega^i(M)$ is a chain complex. We define the $i$-currents to be elements of the topological dual space to $\Omega^i(M)$. We usually use $\partial$ to denote the differential on the dual complex. A 2-current \emph{carried by $\F$} is a 2-current which lies in the closure of the set of 2-currents represented by subsurfaces of leaves of $\F$.

		$i$-currents behave like $i$-dimensional submanifolds; one can evaluate $i$-forms on them, their boundaries are $(i-1)$-currents, and Stokes' theorem holds. 
	\end{subsection}
	\begin{subsection}{Smoothness of heat kernels}
		\newcommand{\Ko}{K_{\text{approx}}}
	\newcommand{\Err}{\mathcal E}
	\newcommand{\X}{\partial_t - \Delta_g}
	\newcommand{\fspace}{\mathcal X}
	\newcommand{\poly}{\text{poly}}
	\newcommand{\gaussianvol}{\frac {\sqrt {g(x)}}{(4\pi (t-t_0))^{d/2}}}
	\newcommand{\rr}[1]{\text{inj}(#1)}
	\newcommand{\gaussian}[4]{\exp \left(-\frac{g(x)_{ij}({#1}-{#2})^{\otimes 2, ij}}{4(t_{#4}-t_{#3})}\right)} 
	\newcommand{\simplegaussian}{\exp \left(- \frac{g_{ij}(x) (x-y)^{\otimes 2,ij}}{4t}\right)}
	\newcommand{\simplegaussianbar}{\exp \left(- \frac{g_{ij}(x) (\overline x-y)^{\otimes 2,ij}}{4t}\right)}
	\newcommand{\upsfact}{\upsilon(x)}
	\newcommand{\kapprox}{k_{\text{approx}}}
	\newcommand{\intboundsA}{\substack{0\leq t_0 \leq t\\ x \in \Omega}}
	\newcommand{\dta}{(t_{\alpha+1}-t_{\alpha})}
	\newcommand{\upsxy}{ \upsilon\left(\frac{\abs{x-y}}{\rr{y}}, \frac{t-t_0}{\rr{y}^2}\right)}
	\newcommand{\sspace}{\mathcal Y}

\plabel{fix:13.2}{In this section, we prove the following proposition which will be useful in the proof of \cref{lem:smooth}.} 

\begin{prop}\label{prop:perturbkernel}
	Let $\Omega$ be the unit disk in $\R^d$. Let $\set{g_\eps}_{\eps\in (-1,1)}$ be a $1$-parameter family of metrics on a neighbourhood of $\Omega$. Let $k_\eps(x,y,t):\Omega \times \Omega \times (0,\infty) \to \R$ be the heat kernel for the metric $g_\eps$ with Dirichlet boundary condition 0. Then for any fixed $t>0$, the derivatives of $k_\eps(x,y,t)$ with respect to $\eps$ exist and are continuous as long as the corresponding derivatives of $g_\eps$ exist and are continuous.
\end{prop}
Before beginning with the proof of \cref{prop:perturbkernel}, we set up some notation.

Let $\fspace$ be the space of smooth functions on $[0,\infty) \times \Omega$ which vanish on $\partial \Omega$. Let $\sspace$ be the space of smooth functions on $\Omega \times \Omega \times [0,\infty)$ which extend smoothly to a neighbourhood of $\Omega$. The $[0,\infty)$ factor represents time. We will write $g$ for $g_\eps$ and use the shorthand $\sqrt g = \sqrt {\det g_{ij}}$. We will use Einstein summation notation throughout this section. The coordinates of vectors will be written with upper indices. Given a vector $x$, we use $x^{\otimes n, i\dots j}$ to denote the $n^{th}$ tensor power of $x$ with indices $i,\dots,j$. For example, if $x$ is a vector, then $x^{\otimes 2, ij}$ is the tensor square of $x$ with indices $i$ and $j$. We will use $|\cdot|$ to denote distance in the standard Euclidean metric.

Let $h$ be a positive number strictly smaller than the smallest eigenvalue of $g_{ij}(x)$ for all $x\in \Omega$. For any $x\in \Omega$ at Euclidean distance less than $\sqrt h$ to $\partial \Omega$, there is a unique closest point to $x$ on $\partial \Omega$. Call the reflection of $x$ through this closest point $\overline{x}$. Here, ``closest'' and ``reflection'' are taken not with respect to the Euclidean metric, nor with respect to the Riemannian metric $g_{ij}$, but with respect to the flat metric induced by the inner product $g_{ij}(x)$.

Given a smooth function $k:\Omega \times \Omega \times (0,\infty)\to \R$, we say that $k$ is \emph{Gaussian-type} if $k$ can be written as
\begin{equation}\label{eqn:gaussiantype}
	k(x,y,t) = c(x,y,t) \frac{(x-y)^{\otimes a}}{t^{d/2+b}} \simplegaussian
\end{equation}
or
\begin{equation}
	k(x,y,t) = c(x,y,t) \frac{(\overline x-y)^{\otimes a}}{t^{d/2+b}} \simplegaussianbar
\end{equation}
for some non-negative integers $a$ and $b$ and a smooth function $c(x,y,t)$ in $\sspace$.

We say that the \emph{order} of $k$ is $a - 2b+1$. A Gaussian-type function of high order has a relatively mild singularity at $x=y, t=0$.

\begin{prop}\label{prop:diff}
If $f$ is a Gaussian-type function of order $k$, then $\partial_t f$ is Gaussian-type of order $k-2$ and $\partial_x f$ and $\partial_y f$ are Gaussian-type of order $k-1$. The derivative of $f$ with respect to $g$ is of order $k$.
\end{prop}

\begin{prop}\label{prop:subcriticalint}
If $k(x,y,t)$ is Gaussian-type of non-negative order and $f$ is a bounded function, then $$\abs{\int_{\Omega} k(x,y,t)f(t,x) dx} \leq \frac C {\sqrt{t}}$$ for some constant $C$.
\end{prop}
\begin{proof}
	Suppose $k$ has the form of \cref{eqn:gaussiantype}. Let $C$ be an upper bound for $\abs{c(x,y,t)}$ over $\Omega \times \Omega \times [0,t]$. Recall that $h$ is a small enough constant that $g_{ij}-hI$ is positive definite everywhere. Let $F$ be an upper bound for $\abs{f}$. We will make the substitution $r=\abs{x-y}$ in the integral below:
\begin{align*}
	\abs{\int_\Omega k(x,y,t) f(t,x)dx} & \leq \int_\Omega CF \frac{\abs{x-y}^a}{t^{d/2+b}} \exp\left(-\frac{h\abs{x-y}^2}{4t}\right)dx\\
	&<\int_0^\infty CF \frac{r^{a+d-1}}{t^{d/2+b}} \exp\left(-\frac{hr^2}{4t}\right) dr\\
	&\propto t^{a/2-b}
\end{align*}
If the order of $k$ is non-negative, then $a/2-b\geq -1/2$ as desired. In performing the integral over $r$, we used the Gaussian integral
\begin{dmath*}
	\int_0^\infty r^k\exp \left(-\frac {r^2} {t}\right)\,dr \propto t^{(k+1)/2}
\end{dmath*}
where the constant of proportionality depends only on $k$. The same argument works for Gaussian-type functions of the second kind.
\end{proof}

\begin{prop}\label{prop:sqrtestimate}
	Let $$Z(n)=
	\int_{0 \leq t_0 \leq \hdots \leq t_n = 1 } \prod_{\alpha=0}^{n-1}\left(\frac 1 {\sqrt{t_{\alpha+1}-t_\alpha}}\right) dt_0\hdots dt_{n-1}$$
	The integral converges and the explicit bound $$Z(n)<\frac {10^{n}}{\sqrt {n!}}$$ holds.
\end{prop}

\begin{proof}
	We work by induction. The case $n=0$ holds trivially. Note that $Z(n)$ has the following scaling property:
	$$
	\int_{0 \leq t_0 \leq \hdots \leq t_{n} = T} \prod_{\alpha=0}^{n-1}\left(\frac 1 {\sqrt{t_{\alpha+1}-t_\alpha}}\right) dt_0\hdots dt_{n-1} = \sqrt{T^{n}} Z(n)$$
	This is because the measure scales like $T^{n}$ and the integrand scales like $\frac 1 {\sqrt{T^{n}}}$.
	
	Now assume that the result holds for $Z(n)$. We now separate integration over the final variable $t_{n}$ from integration over the remaining variables in the expression for $Z(n+1)$.

	\begin{align*}
		Z(n+1) &= \int_0^1 \frac 1 {\sqrt {1-t_{n}}} \left(\int_{0\leq t_0\leq ...\leq t_{n}}   
		\prod_{\alpha=0}^{n-1}\left(\frac 1 {\sqrt{t_{\alpha+1}-t_\alpha}}\right) dt_0\hdots dt_{n-1} \right)dt_{n}\\
		 &=\int_0^1 \frac {1} {\sqrt {1 - t_{n}}} \cdot (\sqrt{t_{n}})^nZ(n) dt_{n}\\
		 &< \frac {10^n} {\sqrt {n!}} \int_0^1 \frac {t^{n/2}} {\sqrt {(1-t)}}  dt \tag*{by the induction hypothesis}\\
		 &= \frac {10^n} {\sqrt {n!}} \left(
		 \int_0^{1-1/n} \frac {t^{n/2}} {\sqrt {(1-t)}}dt
		 +\int_{1-1/n}^1 \frac {t^{n/2}} {\sqrt {(1-t)}}dt
		 \right)\\
		 &< \frac {10^n} {\sqrt {n!}} \left(
		 \int_0^{1-1/n} \frac {t^{n/2}}{\sqrt{1/n}}dt
		 +\int_{1-1/n}^1 \frac {1} {\sqrt {(1-t)}}dt
		 \right)\\
		 &< \frac {10^n} {\sqrt {n!}} \left(
		 \int_0^{1} {\sqrt n} {t^{n/2}}dt
		 +\int_{1-1/n}^1 \frac {1} {\sqrt {(1-t)}}dt
		 \right)\\
		 &< \frac {10^n} {\sqrt {n!}} \left(\frac {10} {\sqrt {n+1}} \right)\\
		 &< \frac {10^{n+1}}{\sqrt {(n+1)!}}
	\end{align*}
	This completes the induction.
\end{proof}

\begin{proof}[Proof of \cref{prop:perturbkernel}]

	Our first step is to construct a reasonable approximation to the heat kernel. To do this, we use the method of images. 
	
	Given a source at a point $x\in \Omega$, we add a sink at $\overline x$. The solution to the heat equation on $\R^d$ with this source and sink will be approximately zero on $\partial \Omega$, and therefore is approximately a solution to the heat equation on $\Omega$ with Dirichlet boundary condition 0. Let us make this more quantitative. Let $\upsilon:\Omega \to \R$ be a smooth function satisfying $\upsilon(x)=0$ when $x$ is at Euclidean distance $\geq \sqrt h$ from $\partial \Omega$ and $\upsilon(x)=1$ when $x$ is at Euclidean distance $\leq \sqrt{h}/2$ from $\partial \Omega$. Thus, $\overline x$ is well defined whenever $\upsilon(x)\neq 0$. Let $$k_0(x,y,t)=\frac{\sqrt {g(x)}}{(4\pi t)^{d/2}}\simplegaussian$$ and $$k_1(x,y,t)=\upsfact \frac{\sqrt {g(x)}}{(4\pi t)^{d/2}}\simplegaussianbar.$$

	Thanks to the cutoff function $\upsilon$, $k_1$ is smooth and well defined at all $x \in \Omega$, $y\in \R^d$, and $t>0$. Our first approximation to the heat kernel is $k_0 - k_1$. By symmetry, $k_0-k_1$ vanishes along a hyperplane tangent to $\partial \Omega$. Although $k_0-k_1$ does not vanish for $y\in \partial \Omega$ as required by the Dirichlet boundary condition, we have the following bound:

	\begin{lem}\label{lem:whitney}
		For $x\in \Omega$ and $y\in \partial\Omega$,
		\begin{equation*}
		k_0(x,y,t)-k_1(x,y,t) = c(x,y,t)_{ijk} \frac{{(x-y)}^{\otimes 3,ijk}}{t^{d/2+1}} \simplegaussian
		\end{equation*}
		where $c$ is some smooth tensor valued function each of whose components is in $\sspace$.
	\end{lem}
	\begin{proof}
		Start by comparing $k_0(x,y,t)-k_1(x,y,t)$ to the desired form
		\begin{align}
			c(x,y,t)_{ijk} {(x-y)}^{\otimes 3,ijkl}
			&=\frac{k_0(x,y,t)-k_1(x,y,t)}{\frac 1 {t^{d/2+1}}\simplegaussian} \nonumber \\
			&= \frac{t\sqrt {g(x)}}{(4\pi)^{d/2}} - \upsfact \frac{t\sqrt {g(x)}}{(4\pi)^{d/2}} \exp \left(- \frac{g_{ij}(x) (\overline x-y)^{\otimes 2,ij}-g_{ij}(x) (x-y)^{\otimes 2,ij}}{4t}\right)\label{eqn:lem2a}
		\end{align}

		Since $\overline x$ is a reflection of $x$ in $\partial \Omega$ and $\Omega$ is convex (with respect to the flat metric $g_{ij}(x)$), we have
		\begin{align*}
		g_{ij}(x) (\overline x-y)^{\otimes 2,ij}>g_{ij}(x) (x-y)^{\otimes 2,ij}
		\end{align*}
		Therefore, the exponential term on the right side of \cref{eqn:lem2a} is bounded. In the regime in which where $\abs{x-y}$ is large, say $\geq \sqrt h / 2$, we may use \cref{eqn:lem2a} to choose a smooth candidate for $c(x,y,t)$.

		Now consider the remaining regime where $\abs{x-y}$ is less than $\sqrt h / 2$. Since $y\in \partial \Omega$, the Euclidean distance between $x$ and $\partial \Omega$ is less than $\sqrt h /2$. Therefore, $\upsilon=1$ and the expression in \cref{eqn:lem2a} reduces to

	\begin{equation}\label{eqn:lem2b}
		c(x,y,t)_{ijk} {(x-y)}^{\otimes 3,ijkl}=\frac{t\sqrt {g(x)}}{(4\pi)^{d/2}}\left(1-\exp \left(- \frac{g_{ij}(x) (\overline x-y)^{\otimes 2,ij}-g_{ij}(x) (x-y)^{\otimes 2,ij}}{4t}\right)\right)
	\end{equation}

	It will now be convenient to make a linear change of coordinates so that the midpoint of $x$ and $\overline x$ is at the origin, $x$ lies on the first coordinate axis, and the hyperplane tangent to $\partial \Omega$ at the origin is spanned by the remaining coordinate axes. In this coordinate system, $g_{1i}=0$ for $i\neq 1$. By symmetry, $k_0-k_1$ vanishes whenever $y$ lies on this hyperplane. In this coordinate system, we have
	\begin{align*}
		g_{ij}(x) (\overline x-y)^{\otimes 2,ij}-g_{ij}(x) (x-y)^{\otimes 2,ij} &= g_{11}(\overline x^1-y^1)^2 - g_{11}(x^1-y^1)^2\\
		&= g_{11}(- x^1-y^1)^2 - g_{11}(x^1-y^1)^2\\
		&=4g_{11}x^1y^1\\
		&=O(|x-y|^3)
	\end{align*}
	In the last line we used the constraint that $y\in \partial \Omega$ which implies the geometrical facts that that $x^1=O(|x-y|)$ and $y^1 = O(|y-0|^2) = O(|x-y|^2)$. See \cref{fig:boundary}. Substituting back into \cref{eqn:lem2b}, we find
	\begin{align}\label{eqn:lem2c}
		\frac{k_0(x,y,t)-k_1(x,y,t)}{\frac 1 {t^{d/2+1}}\simplegaussian} &= \frac{t\sqrt {g(x)}}{(4\pi)^{d/2}}\left(1-\exp \left(- \frac{O(|x-y|^3)}{4t}\right)\right)\\
		&=\frac{\sqrt {g(x)}}{(4\pi)^{d/2}}O(|x-y|^3)
	\end{align}
	Thus, this equation defines the desired smooth function $c(x,y,t)$ when $y\in \partial \Omega$.
	\begin{figure}[ht]
		\centering
		\includegraphics[width=0.5\linewidth]{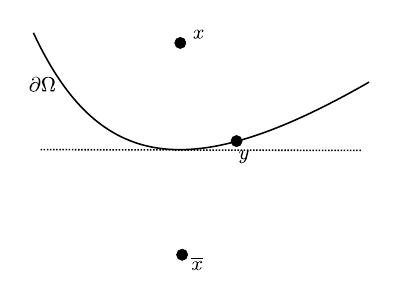}
		\caption{The dotted line is the hyperplane tangent to $\partial \Omega$ at the origin.
		}\label{fig:boundary}
\end{figure}
	\end{proof}

	By the Whitney extension theorem applied to the function $c(x,y,t)$ constructed in \cref{lem:whitney}, we can find a Gaussian-type function $k_2(x,y,t)$ of order $\geq 2$ such that $k_2(x,y,t)=k_0(x,y,t)-k_1(x,y,t)$ whenever $y\in \partial \Omega$. Now define $$\kapprox(x,y,t) = k_0(x,y,t) - k_1(x,y,t)-k_2(x,y,t).$$ By construction, $\kapprox(x,y,t)=0$ whenever $y\in \partial \Omega$. We will use $\kapprox$ as our first approximation to the heat kernel.

	Now we will write down a formal series for the heat kernel. We will then show that the series converges, as do its derivatives with respect to $\eps$. Finally we will check that the series satisfies the Dirichlet boundary condition.

	Define the operator $\Ko:\fspace \to \fspace$ by
	$$\Ko  f (t,y) = \int_{\intboundsA} \kapprox(x,y,t_0) dt_0\, dx$$

	$\Ko$ should be thought of as a first approximation to the heat kernel operator. Let $\Delta_{g}$ be the Laplace-Beltrami operator on $\Omega$ for the metric $g$. In coordinates, the Laplace-Beltrami operator has the form
	$$\Delta_g u = -\frac 1 {\sqrt g} \partial_i \sqrt g g^{ij} \partial_j$$

	Now we define the operator $\Err$ which measures the failure of $\Ko$ to be the true heat kernel operator:
	\begin{equation}
		\Err = Id - (\partial_t - \Delta_g)\Ko \label{eqn:edef}
	\end{equation}
	Rearranging \cref{eqn:edef}, we have
	$$(\partial_t - \Delta_g)\Ko (Id-\Err)^{-1} = Id$$

	Thus, 
	\begin{equation}\label{eqn:heatseries}
		K:=\Ko (Id + \Err + \Err^2 + \hdots)
	\end{equation}
	is a formal series for the heat kernel operator. We now show that the terms $\Ko\Err^n$ are small enough that the sum converges. In what follows, we will use the notation $\poly(\dots)$ to represent an unspecified polynomial in its arguments. When the arguments are tensors, $\poly(\dots)$ may be tensor valued. This polynomial will be allowed to change from line to line as it absorbs constants and the like. The degree and coefficients of this polynomial will crucially not depend on $n$.

	\begin{dgroup*}
		\begin{dmath*}
			\partial_t k_0(x,y,t) = \frac {\sqrt{g(x)}} {(4\pi)^{d/2}}\left(\frac {g(x)_{ij}(x-y)^{\otimes 2, ij}}{4t^{d/2+2}} -\frac{d}{{2t}^{d/2+1}} \right) \simplegaussian\\ + \delta(y-x,t)
		\end{dmath*}
		\begin{dmath*}
			\partial_t k_1(x,y,t) = \upsfact \left( \frac {\sqrt{g(x)}} {(4\pi)^{d/2}}\left(\frac {g(x)_{ij}(\overline x-y)^{\otimes 2, ij}}{4t^{d/2+2}} -\frac{d}{{2t}^{d/2+1}} \right) \simplegaussianbar\\ + \delta(y-\overline x,t) \right)
		\end{dmath*}
	\end{dgroup*}
	Here we are taking distributional derivatives, and $\delta(y-x,t)$ and $\delta(y-\overline x, t)$ are delta distributions.

	Now we apply the Laplace-Beltrami operator in the $y$-coordinate:
	\begin{dgroup*}
		\begin{dmath*}
			\Delta_g k_0(x,y,t) = \frac 1 {\sqrt {g(y)}} \partial_k \sqrt {g(y)} g(y)^{kl} \partial_l  \frac{\sqrt {g(x)}}{(4\pi t)^{d/2}}\simplegaussian
		\end{dmath*}
		\begin{dmath*}
			=\left( \frac {\sqrt{g(y)}} {(4\pi)^{d/2}} \left(\frac{g(y)^{kl}g(x)_{kl}}{2t^{d/2+1}} - \frac{g(x)_{ij}(x-y)^{\otimes ij}}{t^{d/2+2}}\right)\\
			+ \frac{\poly(g,\partial g, \partial\partial g)}{t^{d/2}}
			+ \frac{(x^i-y^i)\poly(g,\partial g,\partial \partial g, x-y)_i}{t^{d/2+1}}\\
			+ \frac{(x-y)^{\otimes ijk}\poly(g,\partial g, \partial \partial g, x-y)_{ijk}}{t^{d/2+2}}
			\right) \simplegaussian
		\end{dmath*}
		\begin{dmath*}
			\Delta_g k_1(x,y,t) = \frac 1 {\sqrt {g(y)}} \partial_k \sqrt {g(y)} g(y)^{kl} \partial_l  \upsfact \frac{\sqrt {g(x)}}{(4\pi t)^{d/2}}\exp \left( -\frac{g(x)_{ij}(\overline{x}-y)^{\otimes ij}}{4t} \right)
		\end{dmath*}
		\begin{dmath*}
			= \upsilon(x) \left( \frac {\sqrt{g(x)}} {(4\pi)^{d/2}} \left(\frac{g(y)^{kl}g(x)_{kl}}{2t^{d/2+1}} - \frac{g(x)_{ij}(\overline x-y^i)^{\otimes ij}}{t^{d/2+2}}\right)\\
			+ \frac{\poly(g,\partial g, \partial\partial g)}{t^{d/2}}
			+ \frac{(\overline x^i-y^i)\poly(g,\partial g,\partial \partial g, \overline x-y)_i}{t^{d/2+1}}\\
			+ \frac{(\overline x-y)^{\otimes ijk}\poly(g,\partial g, \partial \partial g, \overline x-y)_{ijk}}{t^{d/2+2}}
			\right) \simplegaussianbar
		\end{dmath*}
	\end{dgroup*}
	The terms $$\frac {\sqrt{g(x)}} {(4\pi)^{d/2}} \left(\frac{g(y)^{kl}g(x)_{kl}}{2t^{d/2+1}} - \frac{g(x)_{ij}(x-y)^{\otimes 2, ij}}{t^{d/2+2}}\right)$$ and $$\left(\frac{g(y)^{kl}g(x)_{kl}}{2t^{d/2+1}} - \frac{g(x)_{ij}(\overline x-y^i)^{\otimes ij}}{t^{d/2+2}}\right)$$ come from the terms where the derivatives hit only $x-y$ and not $g(y)$. Note that all of the terms are Gaussian-type, and all have non-negative order.
	
	Combining the above, we have

	\begin{dgroup*}
		\begin{dmath*}
			\Err f(t,y) = f(t,y) - (\partial_t - \Delta_g)\int_{\intboundsA} \left(k_0(x,y,t)-k_1(x,y,t)-k_2(x,y,t)\right) f(t_0,x)\, dt_0 \, dx
		\end{dmath*}
		\begin{dmath*}
			=\int_{\intboundsA} 
			\left(\frac{\sqrt{g(x)}}{(4\pi)^{d/2}} \frac{g(x)_{ij}g(y)^{ij}-d}{2(t-t_0)^{d/2}}\\
			+\frac{\poly(g,\partial g, \partial\partial g)}{(t-t_0)^{d/2}}
			+ \frac{(x^i-y^i)\poly(g,\partial g,\partial \partial g, x-y)_i}{(t-t_0)^{d/2+1}} 
			+ \frac{(x-y)^{\otimes 3}\poly(g,\partial g, \partial \partial g, x-y)}{(t-t_0)^{d/2+2}}
			\right) \gaussian xy{0}{} f(t_0,x) \\
			+ \upsilon(d(x,\partial \Omega)) \left(\frac{\sqrt{g(x)}}{(4\pi)^{d/2}} \frac{g(x)_{ij}g(y)^{ij}-d}{2(t-t_0)^{d/2}}
			+ \frac{\poly(g,\partial g, \partial\partial g)}{(t-t_0)^{d/2}}\\
			+ \frac{(\overline x^i-y^i)\poly(g,\partial g,\partial \partial g, \overline x-y)_i}{(t-t_0)^{d/2+1}} \\
			+ \frac{(\overline x-y)^{\otimes 3}\poly(g,\partial g, \partial \partial g, \overline x-y)}{(t-t_0)^{d/2+2}}
			\right)\gaussian xy{0}{} f(t_0,x)\\
			+(\partial_t - \Delta_g) k_2(x,y,t) f(t_0, x)		
			\, dt_0 \, dx
		\end{dmath*}
	\end{dgroup*}
	In the line above, we used that $$\int_{\intboundsA} \delta(x-y, t-t_0)f(t_0, x) dx\, dt_0=f(t,y)$$ and $$\int_\intboundsA \delta(\overline x-y, t-t_0)f(t_0,x)=0.$$ The second equality holds because the delta function can only be nonzero when $x$ and $y$ lie on $\partial \Omega$, but $f$ vanishes on $\partial \Omega$.
	
	We define $e(x,y,t)$ to be the integral kernel for $\Err$ appearing above, so that $$\Err f(y,t) = \int_{\intboundsA} e(x,y,t-t_0) f(t_0, x) \, dt_0\, dx.$$ Now we claim that each term appearing in $e(x,y,t)$ is bounded in absolute value by a Gaussian-type of non-negative order. For the term $$\frac{\sqrt{g(x)}}{(4\pi)^{d/2}} \frac{g(x)_{ij}g(y)^{ij}-d}{2(t-t_0)^{d/2}}\simplegaussian,$$ observe that $g(x)_{ij}g(y)^{ij}-d$ vanishes at $x=y$. Since the partial derivatives of $g$ are bounded, $g(x)_{ij}g(y)^{ij}-d=O(x-y)$ so the term has non-negative order. We can use the same trick to show that the other term involving $g(x)_{ij}g(y)^{ij}$ has non-negative order:
	\begin{align*}
		\MoveEqLeft[10] \abs{\frac{\sqrt{g(x)}}{(4\pi)^{d/2}} \frac{g(x)_{ij}g(y)^{ij}-d}{2(t-t_0)^{d/2}}\simplegaussianbar}\\
		&= \abs{\frac{\sqrt{g(x)}}{(4\pi)^{d/2}} \frac{O(x-y)} {2(t-t_0)^{d/2}}\simplegaussianbar}\\
		&\leq \abs{\frac{\sqrt{g(x)}}{(4\pi)^{d/2}} \frac{O(\overline x-y)} {2(t-t_0)^{d/2}}\simplegaussianbar}
	\end{align*}
	In the last line, we used that $\abs{x-y} < c \abs{\overline x -y}$ for some constant $c$ independent of $x$ and $y$. Finally, we apply \cref{prop:diff} to the term involving $k_2$. Since we constructed $k_2$ to have order 2, $(\partial_t - \Delta_g) k_2(x,y,t)$ is Gaussian-type of order 0 as desired. Let $F$ be an upper bound for $f$ on $\Omega \times [0,t]$. By \cref{prop:subcriticalint}, we have
	\begin{align*}
		\abs{\Err f(t,y)}\leq \frac {CF} {\sqrt t}
	\end{align*}
	for some constant $C$ independent of $y$ and $t$.

	Now we are equipped to handle the more complicated analysis of $\Ko\Err^n$. For notational simplicity, set $x_{n+1}=y$. We have
	\newcommand{\intboundsB}{\substack{0\leq t_0\leq \hdots\leq t_n\leq t\\ x_0,..,x_n \in \Omega}}
	\newcommand{\dxs}{dx_0\,\dots\,dx_n}
	\newcommand{\dts}{dt_0\,\dots\,dt_n}
	\newcommand{\drs}{dr_0\,\dots\,dr_n}
	\newcommand{\deltax}[1]{{x_{\alpha}}^{#1}-{x_{\alpha+1}}^{#1}}
	\newcommand{\deltat}{{t_{\alpha+1}}-t_{\alpha}}
	\begin{align}
		\begin{split}
		|\Ko\Err^n f(t,y)| =& \bigg|
		\int_\intboundsB
		\kapprox(x_n,y,t-t_n)
		\left(
		\prod_{\alpha=0}^{n-1}
		e(x_\alpha, x_{\alpha+1}, t_{\alpha+1}-t_\alpha)
		\right)\nonumber\\
		&f(t_0, x_0)
		\, \dxs \, \dts
		\bigg|
		\end{split}\label{eqn:ken}\\
		\leq& \int_{0\leq t_0 \leq \hdots \leq t_n \leq t} C_1 \left(\prod_{\alpha=0}^{n}\frac{C_2}{\sqrt {t_{\alpha+1}-t_\alpha}} \right) F \dts\\
		\leq& \frac {C^n}{\sqrt {n!}}\label{eqn:ken2}
	\end{align}
	for a large enough constant $C$ possibly depending on $d$, $g$, $f$, and $t$, but not on $n$. In \cref{eqn:ken}, we used \cref{prop:subcriticalint} on both $\kapprox$ and $e$. In \cref{eqn:ken2}, we used \cref{prop:sqrtestimate}. It follows that the formal series \cref{eqn:heatseries} converges uniformly and is indeed the heat kernel operator. The Dirichlet boundary condition is satisfied since $\kapprox(x,y,t)$ vanishes whenever $y\in \partial \Omega$.

	Finally, we need to check that the derivatives of $\Ko(1+\Err + \Err^2 \dots)$ with respect to variations of $g$ exist. Recall from \cref{prop:subcriticalint} that differentiating a Gaussian-type function with respect to $g$ does not change its order. Therefore, the derivatives of $\Ko\Err^n f$ with respect to $g$ have the same form as in \cref{eqn:ken} except that $\kapprox$ and $e$ are replaced with new Gaussian-type functions of the same order. Therefore, a similar bound applies and the series for the derivatives of $\Ko(1+\Err+\Err^2 \dots)$ converges. The corresponding integral kernels are the desired derivatives of the heat kernel with respect to variations of $g$.
\end{proof}

	\end{subsection}
	\begin{subsection}{Assorted notation}
		Given points $x,y$ on the universal cover of a leaf $\lambda$, let $d(x,y)$ denote the leafwise Riemannian distance between $x$ and $y$ in $\ucover \lambda$. We use $\star_2$ to mean the leafwise Hodge star operator acting on $\Lambda^*(T^*F)$.
	\end{subsection}

\end{section}

\begin{section}{Existence of transverse Reeb flows}\label{sec:almostharmonic}
	In this section, we prove \cref{thm:main}. For the rest of this section, we fix a closed, oriented 3-manifold $M$ with a co-oriented $C^2$ foliation $\F$. As discussed in \cref{sec:prel1}, we can assume that $\F$ is $C^{\infty,2}$. We also fix a Riemannian metric on $M$.

	\plabel{fix:9.1}{First we explain the claim from the introduction that an invariant transverse measure is an obstruction to a transverse Reeb flow.} Suppose $\F$ supports an invariant transverse measure. As discussed in \cite{ruelle_currents_1975}, $M$ then has a non-trivial closed 2-current $\Sigma$ carried by $\F$. Suppose that $\alpha$ is a contact form on $M$ with Reeb flow transverse to $\F$. Then $d\alpha >0$ on $T\F$. Using Stokes' theorem for 2-currents (see for example \cite{dinh_introduction_2005} or \cite{sullivan_cycles_1976}), we then have $$0 > \int_\Sigma d\alpha =\int_{\partial \Sigma} \alpha = 0.$$
	This is the desired contradiction.

	For the rest of the section, we consider the case that $\F$ supports no transverse invariant measure. In particular, this implies that $\F$ is taut. \plabel{fix:strategy}{Our plan is to start with a choice of $C^{\infty,1}$ transverse measure $\tau$, and then use a diffusion operator convert $\tau$ into a log superhamonic transverse measure. The diffusion operator depends on our choice of background Riemannian metric. Finally, we will give a recipe to write down a linearly perturbation of a log superharmonic transverse measure into a contact structure with Reeb flow transverse to $\F$.}



	\begin{subsection}{Structure theorem for \texorpdfstring{$C^2$}{C2} foliations}
		Deroin and Kleptsyn gave a precise picture of the long term dynamics of leafwise Brownian motion and the holonomy along such paths. We state their main theorem here for reference:
		\begin{thm}[\cite{deroin_random_2007}]\label{thm:structure}
			Let $\F$ be a $C^1$ foliation of a closed 3-manifold $M$. Then either $\F$ supports an invariant transverse measure, or $\F$ has a finite number of minimal sets $\mathcal M_1,\dots,\mathcal M_k$ equipped with probability measures $\mu_1,\dots,\mu_k$, and there exists a real $\kappa > 0$ such that:
			\begin{enumerate}
				\item \textbf{Contraction.} For every point $x\in M$ and almost every leafwise Brownian path $\gamma$ starting at $x$, there is an orthogonal transversal $I_\gamma$ at $x$ and a constant $C_\lambda > 0$, such that for every $t>0$, the holonomy map $h_{\gamma|_{[0,t]}}$ is defined on $I_\gamma$ and $$|h_{\gamma|_{[0,t]}}(I_\gamma)| \leq C_\gamma\exp(-\kappa t)$$
				\item \textbf{Distribution.} For every point $x\in M$ and almost every leafwise Brownian path $\gamma$ starting at $x$, the path $\gamma$ tends to one of the minimal sets, $\mathcal M_j$, and is distributed with respect to $\mu_j$, in the sense that $$\lim_{t\to \infty} \frac 1 t \gamma_* \leb_{[0,t]} = \mu_j$$ where $\leb_{[0,t]}$ is the standard Lebesgue measure on $[0,t]$.
				\item \textbf{Attraction.} The probability $p_j(x)$ that a leafwise Brownian path starting at a point $x$ of $M$ tends to $\mathcal M_j$ is a continuous leafwise harmonic function (which equals 1 on $\mathcal M_j$).
				\item \textbf{Diffusion.} When $t$ goes to infinity, the diffusions $D^t$ of a continuous function $f:M\to \R$ converge uniformly to the function $\sum_j c_jp_j$, where $c_j=\int f \, d\mu_j$. In particular, the functions $p_j$ form a base in the space of continuous leafwise harmonic functions.
			\end{enumerate}
		\end{thm}
		\begin{rmk}
			It might be surprising that holonomy is contracting in almost every direction. It is instructive to verify \cref{thm:structure} for the foliation in \cref{ex:3}. A point undertaking Brownian motion in a pair of pants in $\F$ has three options for a cuff through which to exit. Two of these options have contracting holonomy and one has expanding holonomy. Therefore, with overwhelming probability, holonomy along leafwise Brownian path is exponentially contracting. In this case, there is only one minimal set $\mathcal M_1=M$. The ergodicity statement reduces to the ergodicity of the associated dynamical system on the $S^1$ factor transverse to the leaves of $M^\circ$ defined by
			\begin{equation*}
				z \mapsto
				\begin{cases}
					\sqrt z, & \text{with probability 1/2} \\
					-\sqrt z, & \text{with probability 1/2}
				\end{cases}
			\end{equation*}
			where $S^1$ is identified with the unit circle in $\C$.
		\end{rmk}
		\begin{rmk}
			The non-existence of an invariant transverse measure is equivalent to an isoperimetric inequality for subsurfaces of leaves, i.e. the existence of a Cheeger constant $h>0$ such that $\frac{\abs{\partial S}} {\abs{S}} \geq h$ for any compact subsurface $S$ of the leaves of $\F$ \plabel{fix:10.1}{\cite{goodman_holonomy_1979}, \cite[Example 7.6]{calegari_foliations_2007}}. Therefore, the theorem above may be regarded as an analogue of the Cheeger inequality for foliations: if $\F$ has a nonzero Cheeger constant, then leafwise random walks converge quickly to a stationary distribution.
		\end{rmk}

		Our foliations are $C^2$, so we can upgrade the contraction result to an infinitesimal version:
		\begin{prop}\label{prop:contraction}
			Let $\F$ be a $C^2$ foliation of a closed 3-manifold $M$. Suppose further that $\F$ does not support an invariant transverse measure. Then there exists a real number $\kappa > 0$ such that for every point $x\in M$ and almost every leafwise Brownian path $\gamma$ starting at $x$, there is a constant $C_\gamma> 0$ such that for every time $t>0$, the holonomy map $h_{\gamma|_{[0,t]}}$ satisfies 
			\begin{equation}\label{eqn:contractionbound}
				\abs{h_{\gamma|_{[0,t]}}'} \leq C_\gamma\exp(-\kappa t).
			\end{equation}
			Moreover, $C_\gamma$ can be chosen to be a Wiener measurable function of $\gamma$.
		\end{prop}
		\begin{rmk}
			We don't need to specify a transverse measure for the norm $\abs{\cdot}$ since $C_\gamma$ can absorb constant factors.
		\end{rmk}
		\begin{proof}
			Given an orientation preserving homeomorphism $g$ between two intervals $I_0$ and $I_1$, we define the distortion of $g$ by$$\distort(g)=\sup_{a,b,c,d \in I_0}
			\frac {a-b}{c-d} \bigg/ \frac {g(a)-g(b)}{g(c)-g(d)}.
			$$
			The distortion is equal to 1 if and only if $g$ is linear. It is submultiplicative with respect to composition of functions. If $g$ is $C^2$, then there is a bound on the distortion in terms of $g''$:
			\begin{align}
				\distort(g) &= \frac{\sup_{x\in I_0} g'(x)}{\inf_{x\in I_0} g'(x)}\\
				&\leq 1 + \frac{\abs{I_0}\sup_{x\in I_0}{\abs{g''(x)}}}{\inf_{x\in I_0} \abs {g'(x)}}\label{eqn:ineq}
			\end{align}

			Let $\varepsilon$ be small enough that holonomy maps over leafwise paths of length $\leq 1$ are defined on all orthogonal transversals of length $\leq \varepsilon$. Let $A$ be an upper bound for the quotient $\frac{\sup_{x\in I_0} \abs{g''(x)}} {\inf_{x\in I_0} \abs{g'(x)}}$ over all holonomy maps $g$ for orthogonal transversals of length $\leq \varepsilon$ along paths of length $\leq 1$.

			Let us now consider a leafwise Brownian path $\gamma$ starting at a point $x\in M$. By \cref{thm:structure}, it is \plabel{fix:11.1.1}{almost surely possible} to choose a short orthogonal transversal $I_\gamma$ through $x$ such that holonomy of $I_\gamma$ along $\gamma$ exists for all time, and moreover that there exists a constant $E_\gamma$ permitting the inequality
			\begin{equation}
				\abs{h_{\gamma\mid_{[0,t]}}(I_\gamma)} < E_\gamma\cdot \exp(-\kappa t)
			\end{equation}
			for all $t>0$. \plabel{fix:11.1.2}{Shortening $I_\gamma$ if necessary, we may further assume that} 
			\begin{equation}\label{eqn:Ismall}
				\abs{h_{\gamma\mid_{[0,t]}}(I_\gamma)} < \varepsilon
			\end{equation}
			for all $t>0$.

			Break the holonomy $h_{\gamma|_{[0,t]}}$ into a composition
			$$h_{\gamma|_{[0,t]}} = h_{\gamma|_{[\floor{t},t]}}\circ \dots \circ h_{\gamma|_{[1,2]}}\circ h_{\gamma|_{[0,1]}}.$$ Let $V_i$ be the Riemannian leafwise distance between $\gamma(i)$ and $\gamma(i+1)$. \plabel{fix:11.3}{Since $V_i$ has bounded moments (in particular, bounded second moment), there exists a constant $c$ depending only on the foliation such that} $Pr[V_i > i] \leq \frac c {i^2}$. Thus, 
			\begin{align}
				\sum_i \Pr[V_i > i] &< \sum_i \frac {c}{i^2}\\
				&< \infty
			\end{align}
			It follows from the Borel-Cantelli lemma that with probability one, all but finitely many of the $V_i$ satisfy $V_i<i$. So there exists $i_0$ depending on $\gamma$ so that $V_i < i$ for all $i > i_0$. 

			\plabel{fix:11.2}{The following sequence of inequalities holds with probability one.}
			\begin{align}
				\distort(h_{\gamma|_{[0,t]}}) &< \prod_{i=0}^{\floor t} \distort(h_{\gamma|_{[i,\min(i+1,t)]}})\label{eqn:submul}\\
				&< \prod_{i=0}^{\floor t} \paren{1 + A\abs{h_{\gamma\mid_{[0,i]}}(I_\gamma)}}^{\ceil{V_i}} \label{eqn:l2}\\
				&< \prod_{i=0}^{\floor t} (1 + A\cdot E_\gamma \exp(-\kappa i))^{\ceil{V_i}}\label{eqn:l3}\\
				&< \left(\prod_{i=0}^{i_0} (1 + A\cdot E_\gamma \exp(-\kappa i))^{\ceil{V_i}} \right)\\
				&\qquad \left(\prod_{i=i_0+1}^{\floor t} (1 + A\cdot E_\gamma \exp(-\kappa i))^i \right)\\
				&<B_\gamma
			\end{align}
			where $B_\gamma$ is a constant which can be chosen independent of $t$. In \cref{eqn:submul} we invoked submultiplicativity of distortion. In \cref{eqn:l2}, we used the fact that $\gamma|_{[i,i+1]}$ is homotopic rel. endpoints to the concatenation of at most $\ceil{V_i}$ paths of length at most 1 and invoked \cref{eqn:ineq}. We also used \cref{eqn:Ismall} to guarantee that the bound involving $A$ applies.

		 Now we can estimate $h'_{\gamma|_{[0,t]}}$ using our bound on the distortion of $h_{\gamma|_{[0,t]}}$ combined  with the macroscopic bound given by \cref{thm:structure}:
			\begin{align*}
				\abs {h'_{\gamma|_{[0,t]}}} &< B_\gamma \abs{h_{\gamma|_{[0,t]}}(I_\gamma)}\\
				&< B_\gamma \cdot E_\gamma \exp(-\kappa t)
			\end{align*}
			So we may take the constant in \cref{eqn:contractionbound} to be $B_\gamma E_\gamma$.

		\plabel{fix:16.1}{It remains to check that we can make the choice of $C_\gamma$ a Wiener measurable function of $\gamma$.} Define the $\delta$ neighbourhood of $\gamma$, denoted $N_\delta(\gamma)$, to be the set of paths $$N_\delta(\gamma)=\set{\gamma_1 \mid \forall t:\, d(\gamma_1(t),\gamma(t)) < \delta}.$$

		Set $\delta = \min(1,R/10)$ where $R$ is a lower bound for the injectivity radius of leaves of $\F$. Then for any $\gamma_1 \in N_\delta(\gamma)$, $\gamma_1$ is homotopic rel endpoints to a concatenation $\gamma_2 \circ \gamma$ where $\gamma_2$ has length less than 1. For any $t>0$, if $$\abs{h_{\gamma|_{[0,t]}}'} \leq C_\gamma\exp(-\kappa t),$$ then for any $\gamma_1\in N_\delta(\gamma)$ we have 
		\begin{align}
			\abs{h_{\gamma_1|_{[0,t]}}'} &\leq \abs{h_{\gamma_2}'}\abs{h_{\gamma|_{[0,t]}}'}\\
			&\leq AB_\gamma E_\gamma\exp(-\kappa t)
		\end{align}
		Therefore, $A B_\gamma E_\gamma$ is a uniform choice for $C_{\gamma_1}$ that makes \cref{eqn:contractionbound} work for any $\gamma_1 \in N_\delta(\gamma)$.

		Let $\Gamma$ be a countable set of paths such that the sets $N_\delta(\gamma)$ for $\gamma\in\Gamma$ constitute a cover of the path space. Observe that $N_\delta(\gamma)$ is a Wiener-measurable set. Define $$C_{\gamma_1} = \sup_{\substack{ \gamma\in \Gamma\\\gamma_1 \in N_\delta(\gamma)}} A B_\gamma E_\gamma.$$ This choice of $C_{\gamma_1}$ satisfies \cref{eqn:contractionbound} and is Wiener measurable as a countable supremum of Wiener measurable functions.
		\end{proof}
	\end{subsection}

	\begin{subsection}{Logarithmic diffusion}
		In this subsection, $x$ will be a point in $M$ and $\lambda$ will be the leaf containing $x$. \plabel{fix:8.2.2}{Since we assumed $\F$ to be taut and not equal to the foliation of $S^1\times S^2$ by spheres, the universal cover $\widetilde \lambda$ is a properly embedded plane in $\widetilde M$.} We will abbreviate $f_{\tau,\lambda,b}$ to $f_\tau$ when there is no danger of ambiguity. Recall that $W_x$ is defined to be the distribution of leafwise Brownian paths starting at a point $x$. Given a time $T>0$, we define a diffusion operator $\DD_T$ acting on $C^{\infty,1}$ transverse measures of full support by
		\begin{align*}
			\DD_T(\tau)|_x&=\exp\paren{\Egamma{ \log\abs{h'_{\gammaT T}}_\tau}}\tau|_x
		\end{align*}
		Here, $\E_{\gamma\sim W_x}$ means that we are taking the expectation over paths $\gamma$ sampled from $W_x$. Also recall that $\abs{h'_{\gammaT T}}$ is the infinitesimal holonomy along $\gamma$ from $\gamma(0)$ to $\gamma(T)$.

		In other words, the $\DD_T(\tau)$ length of an infinitesimal transverse arc through $x$ is the geometric mean of the $\tau$ lengths obtained by holonomy transport along time $T$ leafwise random walks. For an individual leaf $\lambda$, the function $\log f_{\tau,\lambda,b}$ on $\ucover \lambda$ evolves according to the standard heat flow. If the geometric mean were to be replaced with an arithmetic mean, we would obtain the leafwise heat flow defined by Garnett~\cite{garnett_foliations_1983}. Unlike the heat flow operator, our diffusion operator does not conserve mass. Its advantage is that it gives greater weight to the well-behaved smaller holonomies, and therefore allows us to prove a quantitative convergence result.

		The main result of this subsection is \cref{prop:superharmonic}, which asserts that at some large but finite time $T$, the Radon-Nikodym derivatives of the diffused transverse measure are exponentials of strictly superharmonic functions. This implies, for example, that the distance to a nearby leaf, as measured by the diffused transverse measure, never has local minima. 

		As written, it is hard to prove any transverse regularity for the diffused measure. One cannot compare the diffused transverse measure at two nearby points on distinct leaves because holonomy of a transversal connecting these two points typically blows up in finite time along some long paths. We resolve this by introducing a cutoff function $\phi_{R,S}$. \plabel{fix:12.1}{We set}
		$$\DD_{T,R,S}(\tau)|_x=\exp\paren{\Egamma{ \log\abs{h'_{\gammaT T}}_\tau \phi_{R,S}\paren{x,\gamma(t_{\far})}}}\tau|_x$$

		where $t_{\far} \in [0,T]$ is the time minimizing $\phi_{R,S}(x,\gamma(t_\far))$ and $\phi_{R,S}(x,-)$ is a smooth cutoff function defined on $\ucover \lambda$ satisfying the following conditions:
		\begin{itemize}
		\item $\phi_{R,S}(x,y)=1$ when $d(x,y)<R$
		\item $\phi_{R,S}(x,y)=0$ when $d(x,y)>SR$
		\item $\phi$ has all derivatives bounded in absolute value by $\frac {10} {S}$
		\item The superlevel sets of $\phi(x,-)$ on $\ucover \lambda$ are topological disks.
		\end{itemize}

		The reader may now proceed to the proof of \cref{prop:superharmonic} and refer to the technical lemmas below as needed.

		\begin{lem}\label{lem:smooth}
			If $\tau$ is a $C^{\infty,1}$ transverse measure, then $\DD_{T,R,S}(\tau)$ is a $C^{\infty,1}$ transverse measure for all $T,R,S$.
		\end{lem}
		\begin{proof}
			This is where the cutoff comes in handy. Let $x$ be a point on a leaf $\lambda$. In $\ucover M$, choose a neighbourhood $U$ of the radius $SR$ disk in $\ucover \lambda$ centred at $x$ which is skinny enough in the transverse direction that it is foliated as a product. In a neighbourhood of $x$, the values of $\DD_{T,R,S}(\tau)$ depend only on information in $U$. If we were to ignore the dependence of $t_\far$ on $\gamma$, then smoothness of $\DD_{T,R,S}(\tau)$ would follow from standard results regarding the smoothness of heat kernels with respect to compact variations of the metric. The dependence of $t_\far$ on $\gamma$ can indeed be trivialized, as we now detail.

			Suppose $x(\eps)$ is a smooth path in $U$ with $x(0)=x$. To show that $\DD_{T,R,S}(\tau)$ is once differentiable at $x$, we need to show that $$\frac{\partial}{\partial \eps} \left(\DD_{T,R,S}(\tau)(v)\mid_{x(\eps)}\right)$$ exists for any smooth vector field $v$.

			Let $\ucover \lambda_\eps\subset U$ be the leaf containing $x(\eps)$. \plabel{fix:13.1}{Let $\psi_\eps:\R^2 \to \ucover \lambda_\eps$} be a smooth family of diffeomorphisms parameterized by $\eps$ such that $\psi_\eps(0)=x(\eps)$ and that $\psi_\eps^*(\phi(x(\eps),-))$ is a standard, rotationally symmetric function on $\R^2$ independent of $\eps$. Call this standard cutoff function $\phi_0:\R^2\to \R$. We suppress the parameters $R$ and $S$ in this notation since we can take them as constant in the proof of this lemma.

			Let
			\begin{align*}
				g_\eps&=\psi_\eps^*(g)\\
				f^\eps_{\tau}&=\psi_\eps^*(f_{\tau,\lambda_\eps,x(\eps)})
			\end{align*}
			where $g$ is the Riemannian metric on $\ucover M$. With this definition, $f_\tau^\eps$ is a $C^{\infty,1}$ function on $\R^2\oplus \R$, where the $\R$ factor is parameterized by $\eps$. Now we may write
			\begin{align}
				\DD_{T,R,S}(\tau)(v)\mid_{x(\eps)}&=\exp\paren{\E_{\gamma\sim W_{x(\eps)}}\bracket{ \log\abs{h'_{\gammaT T}}_\tau\phi_{R,S}\paren{x,\gamma(t_{\far})}}}\tau(v)|_{x(\eps)}\\
				&=\exp\paren{\E_{\gamma\sim W_{x(\eps)}}\bracket{ \log f_{\tau,\lambda_\eps,x(\eps)}(\gamma(T)) \phi_{R,S}\paren{x,\gamma(t_{\far})}}}\tau(v)|_{x(\eps)}\\
				&=\exp\paren{\int_0^1 \E_{\gamma\sim W_{x(\eps)}}\bracket{\log f_{\tau,\lambda_\eps,x(\eps)}(\gamma(T)) \,\big\vert \, \phi_{R,S}\paren{x,\gamma(t_{\far})} > a  } da}\tau(v)|_{x(\eps)}\\
				&=\exp\paren{\int_0^1 \int_{\Omega(a)} k_{\eps,\Omega(a)}(0,x,T)\log f_\tau^\eps(y) \, dx\,da}\tau(v)|_{x(\eps)}\label{eqn:tmp1}
			\end{align}

			where $\Omega(a)$ is the compact disk $\set{x\in \R^2 \mid \phi_0(x) \geq a}$ and $$k_{\eps, \Omega(a)}(x,y,t):\Omega(a) \times \Omega(a) \times (0,\infty) \to \R$$ is the heat kernel on $\Omega(a)$ for the metric $g_\eps$ with zero boundary condition. By \cref{prop:perturbkernel}, the partial derivative of $k_{\varepsilon, \Omega(a)}(0,x,T)$ with respect to $\eps$ exists and is continuous because $\frac{\partial g}{\partial \eps}$ exists and is continuous. Moreover, by another application of \cref{prop:perturbkernel}, $\frac{\partial}{\partial \eps} k_{\varepsilon, \Omega(a)}(0,x,T)$ varies continuously as we vary $a$. Thus, the partial derivative of \cref{eqn:tmp1} with respect to $\eps$ exists and is continuous.

			A similar argument for higher partial derivatives shows that $\DD_{T,R,S}(\tau)$ has as much regularity as do $g_\eps$ and $f^\eps_\tau$, which is to say $C^{\infty,1}$.

		\end{proof}

		\begin{lem}\label{lem:cutoff}
			For any fixed $T,S$,
			\begin{align}
				\MoveEqLeft \lim_{R\to \infty} \Egamma{ \log f_\tau(\gamma(T))\phi_{R,S}(x,\gamma(\tfar))} \label{eqn:exp1}\\
				&= \lim_{R\to \infty} \Egamma{ \log f_\tau(\gamma(T))\phi_{R,S}(x,\gamma(T))} \label{eqn:exp2}\\
				&=\Egamma{\log f_\tau(\gamma(T))} \label{eqn:exp3}
			\end{align}

			and the convergence for either limit is uniform over $x \in M$ and over $S$. In other words, tail events from abnormally long paths do not contribute much to the expectation. In particular, the right side exists and is continuous.
		\end{lem}
		\begin{proof}
			Say $\gamma$ is a tail event if $d(x,\gamma(t_{far}))>R$. In other words, $\gamma$ is a tail event if $\gamma$ ever travels a distance $R$ away from $x$ before time $T$. The expectations in \cref{eqn:exp1,eqn:exp2,eqn:exp3} differ only on tail events. In order to quantify the contribution of tail events to the expectations, we need to bound both how fast $\log f$ grows on $\ucover\lambda$ and how fast Brownian motion can travel on $\ucover\lambda$. Since $d\log f_{\tau}$ is continuous on $M$, it is bounded above in norm. Therefore, $\log f_{\tau}$ is $L$-Lipschitz on $\ucover\lambda$ for some $L$. In particular, 
			\begin{equation}\label{eqn:growthlimit}
				\abs{\log f_\tau(\gamma(T))}<L\,d(x,\gamma(T))+C
			\end{equation}
			for some constant $C$. 
			
			Let $K$ be a global upper bound for the absolute value of the Gaussian curvature of leaves in $M$. \plabel{fix:15.1}{By the heat kernel estimates from \cite{cheng_upper_1981} or \cite[Theorem B.7.1]{candel_foliations_2000}, we have}
			\begin{equation}\label{eqn:speedlimit}
				\Pr\bracket{d(x,\gamma(\tfar)) > r}< C_0\exp\paren{-\frac {c_0r} {T\sqrt K}}
			\end{equation}
			for some absolute constants $c_0$ and $C_0$.
			
			Note that
			\begin{equation}
			\Pr\bracket{d(x,\gamma(T)) > r} < \Pr\bracket{d(x,\gamma(\tfar)) > r},
			\end{equation}
			so $\Pr\bracket{d(x,\gamma(T)) > r}$ enjoys the same upper bound as in \cref{eqn:speedlimit}.

			Combining \cref{eqn:growthlimit} and \cref{eqn:speedlimit}, the maximum possible contribution of tail events to any of \cref{eqn:exp1,eqn:exp2,eqn:exp3} is bounded above by
			\begin{align}
				\int_R^{\infty} (Lr + C)C_0\exp\paren{- \frac{c_0 r}{T\sqrt K}}\,dr
			\end{align}
			where $L$, $C$, and $K$ are constants depending only on $M$, $\F$, and the metric. This integral converges to zero as $R\to \infty$.
		\end{proof}
		\begin{rmk}
			Here is a heuristic explanation for the bound \cref{eqn:speedlimit}. Brownian motion travels at the usual square root speed on length scales below $\frac 1 {\sqrt K}$, and in negative curvature travels at linear speed at length scales above $\frac 1 {\sqrt K}$. This is because a random walk is exponentially unlikely to backtrack in the presence of negative curvature. A concentration result for the linear speed gives the desired bound.
			\end{rmk}

		In what follows, $\Delta$ denotes the leafwise Laplace operator.
		\begin{lem}\label{lem:derivatives}
			For fixed $T$,
			$$\lim_{R,S\to \infty} \Egamma{\Delta (\log f_\tau(-)\phi(x,-))(\gamma(T))} = \Egamma{\Delta (\log f_\tau)(\gamma(T))}
			$$

			and the convergence is uniform over $M$. In particular, the right side exists and is continuous.
		\end{lem}
		\begin{proof}
			This follows from an argument parallel to that in the proof of \cref{lem:cutoff}. \plabel{fix:15.2}{We just need to check that $\Delta (\log f_\tau(-)\phi(x,-))$ grows slowly enough on $\ucover \lambda$.} We have
			\begin{align*}
				\Delta (\log f_\tau(-)\phi(x,-)) &= (\Delta \log f_\tau) \phi(x,-) + \langle d\log f_\tau, d\phi(x,-)\rangle + \log f_\tau \Delta \phi(x,-)
			\end{align*}
			The terms $\Delta \log f_\tau$ and $d\log f_\tau$ are both bounded because they descend to functions on $M$. The terms $d\phi(x,-)$ and $\Delta \phi(x,-)$ are bounded; in fact, they tend to zero as $S\to \infty$ as arranged in the construction of $\phi$. Finally, $\log f_\tau$ is Lipschitz as noted in the proof of \cref{lem:cutoff}. Therefore, we have a bound of the form 
			\begin{align*}
				\Delta (\log f_\tau(-)\phi(x,-)) (\gamma(T)) <L\,d(x,\gamma(T))+C
			\end{align*}
			for some constants $L$ and $C$ depending only on $M$, $\F$, and the Riemannian metric. The rest of the proof carries through as in \cref{lem:cutoff}.
		\end{proof}

		It will be necessary to understand how $\log f_\tau(x)$ behaves as $x$ undergoes Brownian motion. It\^o's lemma gives the answer:

		\begin{lem}[It\^o's lemma]\label{lem:ito}
			Suppose $x$ evolves according to Brownian motion on a Riemannian manifold $\lambda$ with diffusion rate $\sigma$. Then for a differentiable real valued function $g$ on $\lambda$, $g(x)$ follows a drift-diffusion process with diffusion rate $\sigma \abs{\grad g}$ and drift rate $\frac {\sigma^2}{2}\Delta g$.
		\end{lem}

		\plabel{fix:15.3}{For the reader unfamiliar with the language of It\^o calculus, we provide a restatement in our setting for the case $\sigma=1$:}

		\begin{lem}[It\^o's lemma, reformulated]\label{lem:ito2}
			If $g$ is a real valued $C^2$ function on a Riemannian manifold $\lambda$ and $x\in \lambda$, then for any $y\in \lambda$ and $t_0>0$ we have
			\begin{equation*}
				\lim_{t\to t_0^+} \frac{\Var_{\gamma \sim W_x}[g(\gamma(t)) \mid \gamma(t_0)=y]}{t-t_0} = \abs{\nabla g (y)}
			\end{equation*}
			\begin{equation*}
			\frac{d}{dt} \Egamma{g(\gamma(t))\mid \gamma(t_0)=y}\Bigg|_{t=t_0} = \frac 1 2\Delta g(y)
			\end{equation*}
		\end{lem}

		For more discussion of It\^o's lemma and drift-diffusion process, we point the reader to \cite[Chapter 7]{morters_brownian_2010}.

		\begin{lem}\label{lem:lapint}
			There exists $\kappa > 0$ such that for each minimal set $\mathcal M_i$,
			\begin{equation}\label{eqn:lapint}
				\int_{\mathcal M_i} \Delta \log f_\tau \, d\mu_i < -\kappa
			\end{equation}
			where $\mu_i$ is the probability measure on $\mathcal M_i$.
		\end{lem}
		\begin{proof}
			The basic idea is that, by It\^o's lemma, the integral on the left side of \cref{eqn:lapint} dictates the average drift rate of $\log f_\tau$ for long Brownian paths in $\mathcal M_i$. This drift rate should be negative in accordance with \cref{thm:structure}. We must take a bit more care because, as written, \cref{thm:structure} doesn't give bounds on the expectation of $C_\gamma$.

			Let $\kappa$ and $C_\gamma$ be the constants from \cref{prop:contraction}. Let $g=\Delta \log f_\tau$. Choose $x\in \mathcal M_i$ and let $\gamma \sim W_x$ where $W_x$ is the distribution of leafwise Brownian paths starting at $x$. Choose $C$ large enough that $\Pr_{\gamma\sim W_x}\bracket{C_\gamma < C} > 0$. Let $A$ be the event that $C_\gamma < C$. Let $A_t$ be the event that for all $s\in [0,t]$, 
			\begin{equation}\label{eqn:Cdef}
				f_\tau(\gamma(s))<C\exp(-\kappa s).
			\end{equation}
			By definition, $A=\bigcap_t A_t$. Since $\Pr\bracket{A}>0$, the conditional probabilities satisfy $$\lim_{\substack{i,j\to \infty\\ i>j}}\Pr[A_i \mid A_j]=1.$$ 

			For any $\eps > 0$, \cref{thm:structure} lets us choose a time $dt$ large enough that the time $dt$ diffused function, $D^{dt}g$, satisfies $$\abs{D^{dt}g-\int_{\mathcal M_i}g\, d\mu_i} < \varepsilon$$ at every point in $\mathcal M_i$. For any time $t$, set $t':=\floor{t-dt}$. The essential properties of $t'$ are that it is (locally) independent of $t$, it is a bounded time in the past, and yet is far enough in the past that Brownian motion has had some time to diffuse.

			Now taking the logarithm of the defining constraint \cref{eqn:Cdef}, we have
			\begin{align}
				-\kappa t + \log C &> \E\bracket{\log f_\tau(\gamma(t))\mid A_t}\\
				\begin{split}
				&=\frac 1 {\Pr\bracket{A_t \mid A_{t'}}} 
					\bigg(
					\E\bracket{\log f_\tau(\gamma(t))\mid A_{t'}}\\
					&\qquad- \E\bracket{\log f_\tau(\gamma(t))\mid \neg A_{t}\cap A_{t'}}\Pr\bracket{\neg A_t \mid A_{t'}} 
					\bigg)
				\end{split}\\
				\begin{split}
				&>\frac 1 {\Pr\bracket{A_t \mid A_{t'}}}
					\bigg(
					\E\bracket{\log f_\tau(\gamma(t))\mid A_{t'}}\\
					&\qquad-\E\bracket{\log f_\tau(\gamma(t))-\log f_\tau(\gamma(t'))+\kappa t' - \log C \mid \neg A_{t}\cap A_{t'}} \\
					&\qquad\cdot \Pr\bracket{\neg A_t \mid A_{t'}}
					\bigg)
				\end{split}\\
				&=\frac {\E\bracket{\log f_\tau(\gamma(t))\mid A_{t'}}}{\Pr\bracket{A_t \mid A_{t'}}} + o(t)\label{eqn:simp}
			\end{align}

			Since $\lim_{t\to \infty}\Pr[A_t\mid A_{t'}]\to 1$, we conclude that \begin{equation}\label{eqn:growth}
			\E\bracket{\log f_\tau(\gamma(t))\mid A_{t'}} < -\kappa t + o(t).
			\end{equation}
			\cref{eqn:simp} requires some justification. By \cref{lem:cutoff}, conditioned on any value of $\gamma(t')$, the distribution of $$\log f_\tau(\gamma(t)) - \log f_\tau(\gamma(t'))$$ has a finite expectation. \plabel{fix:17.1}{Moreover, \cref{lem:cutoff}} gives a uniform bound on the tails of this distribution depending only on $t-t'$. Since $\Pr\bracket{ \neg A_t \mid A_{t'}} \to 0$, the term dropped in \cref{eqn:simp} is negligible as claimed.

			On the other hand, we can compute the growth rate of the left side of \cref{eqn:growth} using It\^o's lemma:
			\begin{align*}
				\frac {d\E\bracket{\log f_\tau(t)\mid A_{t'}}}{dt} &= \E\bracket{g(\gamma(t)) \mid A_{t'}}\\
				&= \E\bracket{D^{t-t'}g(\gamma(t')) \mid A_{t'}}\\
				&> \int_{\mathcal M_i} g \, d\mu_i - \eps\\
			\end{align*}
			Here, it was critical for the application of It\^o's lemma that $A_{t'}$ is a piecewise constant function of $t$ so that $\gamma(t) \mid A_{t'}$ is a diffusion process. Comparing with \cref{eqn:growth}, we conclude
			$$\int_{\mathcal M_i} g\, d\mu_i < -\kappa+\eps.$$ Taking $\eps\to 0$ finishes the proof of the lemma.
		\end{proof}

		Now we are ready to prove the main result of this section.
		\begin{prop}\label{prop:superharmonic}
			If $T$, $R$ and $S$ are chosen sufficiently large, then $\DD_{T,R,S}(\tau)$ is log superharmonic. That is, on each leaf $\ucover \lambda$ of $\ucover \F$, the function $\log f_{\DD_{T,R,S}(\tau),\lambda}$ is strictly superharmonic.
		\end{prop}

		\begin{proof}
			The function $\Delta \log(f_\tau)$ is continuous on $M$. By \cref{thm:structure}, as $T\to \infty$, $$\Egamma{\Delta \log(f_\tau)(\gamma(T))}$$ uniformly tends to a linear combination $$\sum_i p_i \int_{\mathcal M_i} \Delta \log f_\tau\, d\mu_i$$ where $\mu_i$ is the probability measure on the $i$\textsuperscript{th} minimal set, $\mathcal M_i$, and $p_i$ are continuous, leafwise harmonic functions on $M$. Choose $\kappa$ satisfying \cref{eqn:lapint} as guaranteed by \cref{lem:lapint}. For large enough $T$, we can guarantee that for any point $x\in M$, 
			\begin{equation}\label{eqn:kappa}
			\Egamma{\Delta \log(f_\tau)(\gamma(T))}<-\frac \kappa 2.
			\end{equation}

			Now we are ready to estimate $\Delta\log f_{\DD_{T,R,S}(\tau)}$.
			\begin{align}
				\lim_{R,S\to \infty} \Delta \log f_{\DD_{T,R,S}(\tau)} &= \lim_{R,S\to \infty}\Delta \Egamma{ \log (f_\tau(\gamma(T)) \phi_{R,S}(x,\gamma(\tfar)))}\\
				&= \lim_{R,S\to \infty}\Delta \Egamma{ \log (f_\tau(\gamma(T)) \phi_{R,S}(x,\gamma(T)))}\\
				\intertext{\rightline{by \cref {lem:cutoff}\qquad}}
				&= \lim_{R,S\to \infty}\Egamma{\Delta (\log (f_\tau(-)) \phi_{R,S}(x,-))(\gamma(T))}\label{eqn:xxx}\\
				\intertext{\rightline{by \cref{eqn:lap}\qquad}}
				&= \Egamma{\Delta(\log f_\tau)(\gamma(T))}\\
				\intertext{\rightline{by \cref{lem:derivatives}\qquad}}
				&< -\frac \kappa 2
				\intertext{\rightline{by \cref{eqn:kappa}\qquad}}
			\end{align}
			\plabel{fix:18.1}{In \cref{eqn:xxx}, the use of \cref{eqn:lap} to commute $\Delta$ with leafwise diffusion was valid because $f_\tau$ and $\phi$ are leafwise $C^\infty$ functions; see also \cref{rmk:fsmooth}.} \cref{lem:cutoff} and \cref{lem:derivatives} further say that the limits above converge uniformly over $x\in M$. Therefore, for large enough $R$ and $S$, $\log f_{\DD_{T,R,S}(\tau)}$ is a strictly superharmonic function on each leaf $\ucover \lambda$ of  $\ucover \F$.
		\end{proof}

	\end{subsection}

	\begin{proof}[Proof of \cref{thm:main}]
		Suppose that $\F$ supports no invariant transverse measure. Then by \cref{prop:superharmonic}, we can find a log superharmonic transverse measure $\tau$. Let $\beta$ be the section of $T^*\F$ defined by $\beta = \star_2 d\log f$, where $\star_2$ is the Hodge star operator on a leaf. The fact that $\log f$ is strictly superharmonic can be written as
		\begin{align*}
		\star_2 d\star_2 d \log f > 0
		\end{align*}
		Therefore $d\beta > 0$ on $T\F$. Furthermore, $\beta \geq 0$ on $\ker(d\log f_\tau)$.

		Choose a vector field $v$ transverse to $\F$ with $\tau(v)=1$, and extend $\beta$ to a 1-form on $M$ by choosing $\beta(v)=0$. Let $\sigma$ be the unit norm 2-vector tangent to $T\F$. Now consider the 1-form $\alpha=\tau + \eps \beta$.
		\begin{align*}
			\alpha\wedge d\alpha &=(\tau + \eps \beta)\wedge d(\tau + \eps \beta)\\
			&= \eps (\tau \wedge d\beta + \beta \wedge d\tau) + O(\eps^2)
		\end{align*}
		Here, $\tau \wedge d\tau$ vanishes because $\tau$ defines a foliation. Evaluating both sides on $v\otimes \sigma$, we get
		\begin{align*}
			(\alpha \wedge d\alpha)(v\otimes \sigma) &= \eps \tau(v)d\beta(\sigma) - \eps(\beta \wedge (\iota_vd\tau))(\sigma) +O(\eps^2)\\
			&= \eps \tau(v)d\beta(\sigma) + \eps(\beta \wedge d\log f_{\tau})(\sigma) +O(\eps^2)
		\end{align*}
		\plabel{fix:18.2}{In the last line, we made use of \cref{prop:iotav}.} By construction, $d\beta(\sigma)>0$ and $\beta \wedge d\log f_{\tau} \geq 0$. Therefore, for small enough $\eps$, the right side is positive everywhere in $M$. So $\alpha$ is a contact perturbation of $\F$.

		Moreover, we have
		\begin{align*}
			d\alpha (\sigma) &= \eps d\beta (\sigma)\\
			&>0
		\end{align*}
		So the Reeb flow of $\alpha$ is transverse to $\F$. Note that $\alpha$ is only $C^1$; if preferred, we can approximate $\alpha$ with a $C^{\infty}$ 1-form having the same properties.
	\end{proof}
	\begin{rmk}
		\plabel{fix:19.1}{We have actually shown that with the hypotheses of \cref{thm:main}, there are 1-forms $\tau,\beta$ such that $\ker \tau=T\F$ and $\tau+\eps \beta$ is a contact structure for all sufficiently small $\epsilon$. This is a linear perturbation in the sense of \cite[Chapter 2]{eliashberg_confoliations_1998}. This gives another proof of \cite[Theorem 2.1.2]{eliashberg_confoliations_1998}.}
	\end{rmk}

\end{section}

\begin{section}{Obstruction 2-currents and Farkas' lemma}\label{sec:farkas}
	In the proof of \cref{thm:main}, we saw that the desired perturbing 1-form $\beta \in \Gamma(T^*\F)$ need only satisfy $d\beta>0$ on $T\F$ and $\beta\geq 0$ on the level sets of $f_\tau$. When $f_\tau$ is log superharmonic, the 1-form $\beta=\star_2 d\log f$ immediately satisfies these conditions. In this section, we give a more flexible criterion for the existence of such a $\beta$ which depends only on the topology of the level sets of $f_\tau$. It has the advantage that for transverse measures one sees in the wild, one can often directly verify the condition and avoid using logarithmic diffusion. \plabel{fix:19.2.1}{The main observation leading to the criterion is that the constraints on $\beta$ are linear, and so can be analyzed via linear programming duality. This idea goes back at least to \cite{sullivan_cycles_1976} where several similar alternatives are proven: either a solution to a system of linear inequalities on $i$-forms exists, or there is an $i$-current furnishing an obstruction. See for example Sullivan's Theorem II.1 and Theorem II.2.}

 	In this section, it will be more convenient to work with reflexive Banach spaces. In \cref{sec:currents}, we introduced the $i$-forms of Sobolev adjusted regularity $(k,l)$. In this section, we fix some large $K>2$. Our $i$-forms will be those of Sobolev adjusted regularity $(K,2)$, and our currents will be elements of the corresponding dual spaces.
	
		\plabel{fix:19.2.2}{Let $\tau$ be an arbitrary non-vanishing $C^{\infty, 1}$ transverse measure. Such a form has adjusted regularity $(\infty,2)$; see \cref{sec:currents}.} Orient the level sets of $f_\tau$ so as to agree with the orientation induced as the boundary of a superlevel set. An \emph{obstruction 2-current for $\tau$} is a 2-current in $M$ which is carried by and co-oriented with $\F$, and whose boundary consists of level sets for $f_{\tau}$ with the negative orientation. For example, a sublevel set near a local minimum for $f_{\tau}$ is an obstruction 2-current. While an arbitrary sublevel set of $f_\tau$ in a leaf $\ucover \lambda\subset \ucover \F$ might appear to be an obstruction 2-current, it typically does not project to a 2-current in $M$ with finite mass.

		\begin{prop}\label{prop:obstruction}
			Either there exists a section $\beta$ of $T^*\F$ with $d\beta > 0$ on $T\F$ and $\beta \geq 0$ on level sets of $f_{\tau}$, or there exists an obstruction 2-current for $\tau$.
		\end{prop}

		Before giving the proof of \cref{prop:obstruction}, we will recall some background from convex optimization. A closed, convex cone in a Banach space is \emph{proper} if it does not contain a line through the origin. We will need a hybrid version of Farkas' lemma with allowances for some strict and some non-strict inequalities. We include a proof because we couldn't find the form we require in the literature.

		\begin{lem}[Farkas' lemma]
			Let $X$, $Y_1$, $Y_2$ be Banach spaces with $X$ reflexive and separable. Let $B\colon Y_1\oplus Y_2 \to X^*$ be a continuous linear map. Let $\inner{\cdot}{\cdot}:(Y_1\oplus Y_2)\otimes X \to \R $ denote the induced pairing. Let $C_i$ be a closed, convex cone in $Y_i$. Suppose further that $C_2$ is proper and that the images $B(C_1)$ and $B(C_2)$ are closed. Then exactly one of the following alternatives holds:
			\begin{enumerate}
				\item There exists $x\in X$ satisfying $\inner{x}{C_1} \geq 0$ and $\inner{x}{C_2 \setminus \set 0} > 0$.
				\item There exists $y_1\in C_1$ and $y_2\in C_2\setminus \set 0$ satisfying $\inner{X}{(y_1,y_2)}=0$.
			\end{enumerate}
.		\end{lem}
		\begin{proof}
			It is clear that the alternatives are mutually exclusive, so we need only prove that at least one of the alternatives holds. We will use the shorthand $B(C)=B(C_1)+B(C_2)$. 

			Suppose that $B(C)=X^*$. Choose any nonzero $(a_1, a_2)\in (C_1,C_2)$. Then choose $(b_1,b_2)\in (C_1,C_2)$ satisfying $B((b_1,b_2))=-B((a_1,a_2))$. The element $(y_1,y_2)=(a_1+b_1, a_2+b_2)$ satisfies alternative two. The fact that $y_2\neq 0$ follows from properness of $C_2$.

			Suppose instead that $B(C)$ does not contain some point $v \in X^*$. By the Hahn--Banach hyperplane separation theorem, there exists a linear functional $x\in X^{**}$ separating $v$ from $B(C)$, in the sense that $x(v)<0$ and $x(B^*(C))\geq 0$. Since $X$ is reflexive, $x$ may be realized as a point in $X$.

			Let $H_x$ be the hyperplane in $X^*$ defined by $x$. Let $A = B(C) \cap -B(C)$. Points in $A$ are ``corners'' of $B(C)$. Observe that $A\subset H_x$. We will now show that it can be further arranged that $H_x \cap B(C)=A$. 

			If $w\in (H_x \cap B(C)) \setminus A$, then another application of the Hahn--Banach theorem gives a linear functional $x'\in X^{**}$ separating $-w$ from $B(C)$. Now $H_{x+x'}\cap B(C)$ is strictly contained in $H_x \cap B(C)$, and does not include $w$. In effect, we have tilted $H_x$ away from $w$. So we have a strategy for removing a single unwanted point from $H_x\cap B(C)$. We will now use Zorn's lemma to show that we can remove all the unwanted points.

			Let $\mathcal S=\set{H_x \cap B(C)}$, where $x$ varies over linear functionals supporting $B(C)$. The elements of $\mathcal S$ are all closed subsets of $B(C)$. The family $\mathcal S$ is partially ordered by inclusion. Since $X$ is separable, any chain in $\mathcal S$ may be refined to a countable chain. Given a countable chain $\set{H_{x_i}\cap B(C)}_{i\in \N}$, let $$\overline x=\sum_i \frac 1 {2^i\norm{x_i}} x_i.$$ With this choice,
			\begin{equation*}
				H_{\overline x}\cap B(C) = \bigcap_{i\in \N} H_{x_i}\cap B(C).
			\end{equation*}
			Therefore, every chain in $\mathcal S$ has a lower bound in $\mathcal S$. By Zorn's lemma, $\mathcal S$ has a minimal element. By the discussion in the previous paragraph, such a minimal element must be equal to $A$.

			Now we have a nonzero candidate $x\in X$ solving the non-strict versions of the inequalities in alternative one. This choice of $x$ satisfies $\inner{x}{(y_1,y_2)}=0$ if and only if $B((y_1,y_2))\in A$. If there exists $a_2\in C_2\setminus \set{0}$ with $B((0,a_2))\in A$, then we may find $(b_1,b_2)\in (C_1,C_2)$ satisfying $B((b_1,b_2))=-B((0,a_2))$. As before, properness of $C_2$ guarantees that $a_2+b_2\neq 0$. Thus, the element $(b_1, a_2+b_2)$ satisfies alternative two. Otherwise, alternative one holds.

		\end{proof}

		\begin{proof}[Proof of \cref{prop:obstruction}]
			Let $X$ be the Sobolev space of sections of $T^*\F$ with Sobolev adjusted regularity $(K,2)$. Since $X$ is a separable Hilbert space, Farkas' lemma will apply. For $i\in \set{1,\,2}$, let $Y_i$ be the space of $i$-currents carried by $\F$. Let $C_1\subset Y_1$ be the closed cone generated by 1-currents contained in and co-oriented with the level sets of $f_{\tau}$. Let $Y_2$ be the space of 2-currents tangent to $\F$ and let $C_2\subset Y_2$ be the closed cone generated by 2-currents co-oriented with $\F$. 

			Although $C_1$ is not proper, $C_2$ is proper. The lack of properness for $C_1$ occurs near critical points of $f_\tau$. At such points, there is a tangent vector $v$ such that both $-v$ and $v$ are limits of oriented subarcs of the level sets of $\tau$. Therefore, $v$ and $-v$ are both elements of $C_1$. On the other hand, if $y\in C_2\setminus \set{0}$, then $-y$ is a 2-current negatively tangent to $\F$ and does not lie in $C_2$.

			Define the pairing $\inner{\cdot}{\cdot}\colon X\otimes (Y_2 \oplus Y_2)\to \R$ by
			\begin{equation*}
				\inner{\beta}{(y_1,y_2)}=\int_{y_2} d\beta + \int_{y_1}\beta. 
			\end{equation*}

			By Stokes' theorem, an element $(y_1,y_2) \in Y_1 \oplus Y_2$ satisfies $\inner{X}{(y_1,y_2)}=0$ if and only if $y_1=-\partial y_2$. So by Farkas' lemma, either there exists a section of $T^*\F$ satisfying the desired properties, or there exists a non-zero 2-current $y_2$ positively tangent to $\F$ such that $\partial y_2$ is negatively tangent to the level sets of $f_{\tau}$. This is exactly an obstruction 2-current.
		\end{proof}

		\begin{prop}\label{prop:obstruction2}
			If $\log f_\tau$ is a strictly superharmonic function on each leaf, then there are no obstruction 2-currents for $\tau$.
		\end{prop}
		\begin{proof}
			This is a generalization of the fact that superharmonic functions have no local minima. Suppose that $S$ is an obstruction 2-current for $\tau$. It is helpful to keep in mind the simplest case of a compact surface $S$ in a leaf $\lambda$ with boundary a negatively oriented union of level sets of $f_{\tau,\lambda}$.

			Strict superharmonicity of $\log f_\tau$ is equivalent to $$\star_2 d \star_2 d\log(f_\tau) < 0,$$ where $\star_2$ is the Hodge star operator on a leaf. So by Stokes' theorem,
			\begin{align*}
				0 &> \int_S d \star_2 d\log f_\tau \\
				&= \int_{\partial S} \star_2 d\log f_\tau\\
				& \geq 0
			\end{align*}
			The last inequality uses the property of obstruction 2-currents that $\partial S$ is a subcurrent of the level sets of $f_\tau$, which are in turn equal to the level sets of $\log f_\tau$. Contradiction.
		\end{proof}

		\begin{ex}[\cref{ex:3} revisited]\label{ex:3revisited}
			Adopt the notation from \cref{ex:3}. Let $T$ be the torus in $M$ along which one can cut to obtain $M^\circ$. Let $\tau=fd\theta$. Let us show that $\tau$ has no obstruction 2-current. Call the purported obstruction 2-current $\Sigma$. It is possible to modify $\Sigma$ so that its boundary lies on $T$. Then $\Sigma$ is a positive combination of horizontal pairs of pants. A pair of pants has two components which are positively oriented level sets of $f$ and one that is a negatively oriented level set of $f$. Therefore, the positively oriented boundary of $\Sigma$ has twice the length of the negatively oriented boundary. Even with cancellation, the positive boundary cannot be empty. Therefore $\Sigma$ cannot be an obstruction 2-current. By \cref{prop:obstruction}, $fd\theta$ may be perturbed to a contact structure with Reeb flow transverse to $\F$.
		\end{ex}

\end{section}

\begin{section}{Questions}
	\begin{itemize}[wide]
		\item Can the results of the present paper be extended to $C^{\infty,1}$ or even to $C^{\infty,0}$ foliations? Many powerful constructions of foliations proceed by iteratively splitting a branched surface~\cite{li_laminar_2002}. The resulting foliation is typically only $C^{\infty,0}$. Kazez--Roberts and Bowden independently showed that the Eliashberg--Thurston theorem does extend to $C^{\infty,0}$ foliations~\cite{kazez_c0_2017,bowden_approximating_2016}. The difficulty in extending our approach is that the direction of expanding holonomy is no longer well defined for $C^{\infty,0}$ foliations. One possible line of attack would be to use the work of Ishii et al. which shows that branched surfaces satisfying a certain handedness condition admit transverse Reeb flows~\cite{ishii_positive_2020}.

		\item What can be said in higher dimensions? One would like to generalize the Eliashberg--Thurston theorem to codimension 1 leafwise symplectic foliations in arbitrary dimension. \cref{prop:superharmonic} and its dependencies all work in arbitrary dimension, giving a log superharmonic transverse measure in the absence of an invariant transverse measure. When can it be upgraded to a log plurisuperharmonic transverse measure?

		\item Is a Reeb flow $R$ transverse to a foliation $\F$ product covered, i.e. conjugate to the standard flow $\frac{\partial}{\partial z}$ on $\ucover M \cong \R^3$? This is equivalent to the statement that $\ucover M/ R$ is topologically an open disk. Since $R$ is transverse to a taut foliation, $\ucover M/R$ is a (possibly non-Hausdorff) 2-manifold. The fact that the flow of $R$ preserves contact planes prohibits certain types of non-Hausdorff behaviour in $\ucover M/R$. In particular, a smooth 1-parameter family of flow lines cannot break into two different families. However, as pointed out to the author by Fenley, there could conceivably still be a sequence of flow lines with more than one limiting flow line.
	\end{itemize}
\end{section}

\printbibliography
\end{document}